\documentclass[12pt]{article}

\usepackage{fullpage}
\usepackage{amssymb}
\usepackage{amsmath}
\usepackage{amsthm}
\usepackage{xypic}
\usepackage{setspace}
\usepackage{array}
\usepackage{url}
\usepackage[colorlinks=true,urlcolor=blue]{hyperref}
\usepackage{pdflscape}

\newtheorem{theorem}{Theorem}[section]
\newtheorem{lemma}[theorem]{Lemma}
\newtheorem{proposition}[theorem]{Proposition}
\newtheorem{corollary}[theorem]{Corollary}

\newcommand{\Z}{\mathbb Z}

\newcommand{\s}{\sigma}
\newcommand{\G}{\Gamma}
\newcommand{\id}{id}

\newcommand{\comment}[1]{}

\newcommand{\ol}[1]{\overline{#1}}
\newcommand{\ul}[1]{\underline{#1}}
\newcommand{\ds}{\displaystyle}
\newcommand{\enant}[1]{{#1}^*}

\DeclareMathOperator{\rank}{rank}

\newcommand{\calK}{\mathcal K}

\newcommand{\calP}{\mathcal P}
\newcommand{\calQ}{\mathcal Q}

\newcommand{\calR}{\mathcal R}

\newcounter{reminder}

\begin{document}

\title{Tight Chiral Polytopes}

\author{Gabe Cunningham\\
Department of Mathematics\\
University of Massachusetts Boston\\
Boston, Massachusetts \\
and \\
Daniel Pellicer \\
Centro de Ciencias Matem\'aticas\\
Universidad Nacional Aut\'onoma de Mexico (UNAM) \\
Morelia, Mexico
}

\date{ \today }
\maketitle

\begin{abstract}
\vskip.1in
\medskip
\noindent

A chiral polytope with Schl\"{a}fli symbol $\{p_1, \ldots, p_{n-1}\}$ has at least $2p_1 \cdots p_{n-1}$ flags, and it is called
\emph{tight} if the number of flags meets this lower bound. The Schl\"{a}fli symbols of tight chiral polyhedra were classified
in an earlier paper, and another paper proved that there are no tight chiral $n$-polytopes with $n \geq 6$. Here we prove that there
are no tight chiral $5$-polytopes, describe 11 families of tight chiral $4$-polytopes, and show that every tight chiral $4$-polytope 
covers a polytope from one of those families.

\comment{
	20B25: Finite automorphism groups of algebraic, geometric, or combinatorial structures.
	51M20: Polyhedra and polytopes; regular figures, division of spaces (under "Real and Complex Geometry")
	52B05: Combinatorial properties of polyhedra (number of faces etc.)
	52B15: Symmetry properties of polytopes
}

\end{abstract}

\section{Introduction}

	An \emph{abstract $n$-polytope} is a partially-ordered set that satisfies many of the properties of
	the face-lattices of convex $n$-polytopes. The maximal chains (called \emph{flags}) are analogous to the
	simplices in the barycentric subdivision of a convex polytope. Automorphisms are order-preserving bijections and are the combinatorial analogue of symmetries of convex polytopes.

	The group of automorphisms of an abstract polytope acts semiregularly on the set of flags, and if the action is transitive
	(and thus regular), then the polytope is said to be regular. This kind of polytopes are regarded as the most symmetric and have been extensively studied. The automorphism group of a regular polytope
	has a standard generating set, and it is possible to recover the polytope from a group in this form, making
	it possible to study regular polytopes completely in terms of their groups.
	
	An abstract polytope is \emph{chiral} whenever the automorphism group has two orbits on
	the flags such that flags that differ in only one element are in opposite orbits. This is the combinatorial analogue to having all symmetry by rotations but none by reflections. As with regular polytopes,
	the automorphism group of a chiral polytope has a standard form, and we can build a chiral polytope
	out of such a group. The study of chiral polytopes grew out from the study of chiral maps and twisted honeycombs 
	(see \cite{coxeter-moser, twisted-honeycombs}), and
	while chiral $3$-polytopes and chiral $4$-polytopes are nowadays plentiful, constructing chiral $n$-polytopes with
	$n \geq 5$ seems to be much harder. 
	To date, there is no known natural family of chiral $n$-polytopes
	with one polytope for each $n$ (whereas there are many examples of families of regular $n$-polytopes, such as $n$-cubes). 
	There is a construction, described in \cite{high-rank-chiral}, that takes a chiral $n$-polytope as input and produces a chiral
	$(n+1)$-polytope, but the polytopes constructed this way are so large that their individual study is out of reach with the current computational means available.
	
	How can we find small examples of chiral polytopes? One strategy is to specify part of the local structure (such as
	what kind of sub-units the polytope is built from) and then use that local structure to put a lower bound on the number of flags.
	This idea was used in \cite{smallest-regular} to find the smallest regular polytopes of each rank, and in \cite{non-flat} to explore bounds in the size of chiral polytopes. A polytope is called \emph{tight} if its number of
	flags is equal to some lower bound. For example, a chiral polyhedron ($3$-polytope) with $p$-gonal faces and $q$ edges at each
	vertex must have at least $2pq$ flags, and so a tight chiral polyhedron has exactly $2pq$ flags (see \cite{tight-polytopes}).
	
	In \cite{tight-chiral-polyhedra}, the first author determined the pairs $(p, q)$ such that there is a tight chiral
	polyhedron with $p$-gonal faces and $q$ edges at each vertex. Furthermore, the first author showed in
	\cite{non-flat} that there are no tight chiral $n$-polytopes with $n \geq 6$. In this work, we exhibit 11 families of tight chiral $4$-polytopes (see Table~\ref{tab:atomic-4}) and show that every
	tight chiral $4$-polytope covers one of the polytopes in these families. Furthermore, we prove the following theorem.
	
	\begin{theorem}\label{t:NoTight5Poly}
	There are no tight chiral $5$-polytopes.
	\end{theorem}

\section{Background}

In this section we summarize relevant definitions and results.

	\subsection{Abstract polytopes}
	
Regular abstract polytopes are a combinatorial generalization of the notion of (geometric) polyhedra explored by Petrie, Coxeter Gr\"unbaum and Dress in the $20$th Century (see \cite{CoxeterSkew}, \cite{DressI}, \cite{DressII}, \cite{GrunbaumOld}). In what follows, we recall the basic definitions. For further details see \cite{arp}.

An {\em abstract polytope} $(\calP,\le)$ of rank $n$ is a partially ordered set satisfying the following four axioms.
\begin{itemize}
 \item[(I)] It has a unique minimal element $F_{-1}$ and a unique maximal element $F_n$.
 \item[(II)] All maximal chains have precisely $n+2$ faces, including $F_{-1}$ and $F_n$. This induces a strictly increasing {\em rank function} $\rank : \calP \to \{-1, \dots, n\}$ where $\rank(F_{-1})=-1$ and $\rank(F_n)=n$.
 \item[(III)] {\em Diamond condition}: Given two elements $F$, $G$ with $\rank(G)=\rank(F)+2$ there exist precisely two elements $H_1$ and $H_2$ with $\rank(H_1)=\rank(H_2)=\rank(F)+1$ such that $F \le H_i \le G$ for $i \in \{1,2\}$.
 \item[(IV)] {\em Strong connectivity}: For any pair of incident elements $\{F,G\} \subseteq \calP$ with $\rank(G)-\rank(F) \ge 3$, the incidence graph of the open interval $(F,G)$ is connected. (The {\em incidence graph} of a partially ordered set has the elements as vertices, and two are adjacent if and only if the corresponding elements are incident.)
\end{itemize}

Throughout this paper we will encounter only abstract polytopes and we shall refer to them simply as `polytopes'. Rank $2$ and $3$ polytopes are also called {\em poylgons} and {\em polyhedra}, respectively. For convenience we refer to the polytope $(\calP,\le)$ simply as $\calP$. Two elements $F,G$ of $\calP$ are said to be {\em incident} if either $F \le G$ or $G \le F$.

The elements of $\calP$ are called {\em faces}. Those of rank $i$ are called {\em $i$-faces}. Following the tradition, the $0$- $1$- and $(n-1)$-faces are called {\em vertices}, {\em edges} and {\em facets}, respectively. For $i \in \{1, \dots, n-2\}$ we define the {\em $i$-skeleton} of $\calP$ as the partially ordered set consisting of all the $j$-faces for $j \le i$. If $F_0$ is a vertex and $F_{n-1}$ is a facet we say that the closed interval $[F_0,F_{n-1}]$ is a {\em medial section} of $\calP$.

The closed intervals of a polytope (also called {\em sections}) satisfy the axioms of abstract polytopes. In particular, any medial section of a polytope is a polytope. The section $[F_0,F_n]$, where $F_0$ is a vertex, is called the {\em vertex-figure} at $F_0$. Every face $F$ may be identified with the section $[F_{-1},F]$ and in this way it may be considered as an abstract polytope.

The maximal chains of $\calP$ are called {\em flags}. Due to the diamond condition, for any flag $\Phi$ and any rank $i \in \{0, \dots, n-1\}$ there exists a unique flag $\Phi^i$ that differs from $\Phi$ precisely in the element of rank $i$. The flag $\Phi^i$ is called the {\em $i$-adjacent flag} of $\Phi$. We extend this notation recursively in such a way that if $w$ is a word on the alphabet $\{0,\dots, n-1\}$ and $i \in \{0, \dots, n-1\}$ then $(\Phi^w)^i = \Phi^{wi}$.

The {\em dual} $\calP^\delta$ of a polytope $\calP$ consists of the same elements as $\calP$ with the partial order reversed. In this way, if $F$ is an $i$-face of an $n$-polytope $\calP$ then it is an $(n-i-1)$-face of $\calP^\delta$.

An $n$-polytope is said to be {\em flat} whenever every vertex is incident to every facet. Given $0 \le k<m \le n$ we say that it is $(k,m)$-flat if every $k$-face is incident to every $m$-face.

There is a unique polytope of rank $0$ and a unique polytope of rank $1$. They correspond to the face lattices of a single point and of a line-segment (with its two endpoints). For each integer $k \ge 2$ there is a unique polygon with $k$ vertices, that corresponds to the face-lattice of a convex $k$-gon. There is also a unique \emph{apeirogon} with infinitely many vertices, corresponding to the face-lattice of the tiling of the real line by unit intervals. Therefore the rank $2$ sections of a polytope are all isomorphic to $k$-gons for some $k$ or to apeirogons.

We say that a polytope is {\em equivelar} if, for every $i \in \{1, \dots, n-1\}$, all sections between an $(i-2)$-face and an incident $(i+1)$-face are $p_i$-gons for some numbers $p_i$, regardless of the choice of $(i-2)$-face and $(i+1)$-face. Regular and chiral polytopes defined below are examples of equivelar polytopes. The {\em Schl\"afli type} (or {\em type} for short) of an equivelar polytope is $\{p_1, \dots, p_{n-1}\}$.

We say that an $n$-polytope $\calQ$ is a {\em quotient} of a polytope $\calP$ whenever there exists a rank and adjacency preserving mapping from the faces of $\calP$ to the faces of $\calQ$. (We say that two $i$-faces are \emph{adjacent} if they are incident to a common $(i-1)$-face and $(i+1)$-face.) In such cases we say that $\calP$ {\em covers} $\calQ$.

An {\em automorphism} of $\calP$ is an order preserving bijection of its faces. The automorphism group is denoted by $\Gamma(\calP)$ and acts freely on the set of flags. It follows from the strong connectivity of $\calP$ that all orbits of flags have the same size $|\Gamma(\calP)|$.

	\subsection{Regularity and chirality}

In this subsection we provide a general background on regular and chiral polytopes.

Our main interest in this paper is on chiral polytopes; hence we shall follow the approach given in \cite{chiral} to the study of the automorphism groups of these two classes of objects, and not the one in \cite{arp} for regular polytopes.

We say that an $n$-polytope $\calP$ is {\em regular} whenever $\Gamma(\calP)$ acts transitively on the set of flags, and it is {\em chiral} whenever $\Gamma(\calP)$ induces two orbits on the flags in such a way that adjacent flags belong to distinct orbits. If $\calP$ is regular or chiral we say that it is {\em rotary}.

For every $i \in \{0,\dots, n-1\}$ the automorphism group of a rotary polytope acts transitively on the $i$-faces. As a consequence, rotary polytopes are equivelar.

It is well-known that for every integers $p_1,\dots,p_{n-1} \ge 2$ there is a regular polytope with type $\{p_1,\dots,p_{n-1}\}$ (see \cite[Chapter 3]{arp}. This is not the case for chiral polytopes, as shown by the following lemma.

\begin{lemma}\label{l:NoChiral2}
If the last entry of the type of a polytope $\calP$ is $2$ then $\calP$ is not chiral.
\end{lemma}

\begin{proof}
If $\calP$ is an $n$-polytope with a $2$ as the last entry of its type then all $(n-3)$-faces belong to precisely two facets. By the diamond condition, also the $(n-2)$-faces belong to two facets. The connectivity of the $(n-2)$-skeleton shows that $\calP$ has precisely two facets and all $i$-faces are incident to them for $i \le n-2$. 

The function that fixes every $i$-face for $i \le n-2$ and interchanges the two $(n-1)$-faces is then an automorphism, and it maps every flag to its $(n-1)$-adjacent. Hence $\calP$ is not chiral.
\end{proof}

Every finite polygon is isomorphic to the face lattice of some convex regular polygon, and hence it is regular. Also the unique infinite $2$-polytope is regular. Hence the rank of a non-regular polytope must be at least $3$. Chiral polytopes exist in ranks $3$ and higher (see \cite{high-rank-chiral}).

All sections of regular polytopes are regular. The facets and vertex-figures of a chiral $n$-polytope may be either regular or chiral; however, the $(n-2)$-faces must be regular (see \cite[Proposition 9]{chiral}). Note that chiral polytopes with chiral facets must have rank at least $4$.

Much of the work on chiral polytopes has been done through a particular presentation of their automorphism groups that we explain next. For another useful presentation see for example \cite{CHOP}.

Given a fixed {\em base flag} $\Phi$ of a rotary $n$-polytope $\calP$ there exist $\s_i \in \Gamma(\calP)$ for $i \in \{1,\dots,n-1\}$ such that $\Phi \s_i = \Phi^{i (i-1)}$. We shall denote the group $\langle \s_1, \dots, \s_{n-1}\rangle$ by $\Gamma^+(\calP)$ and call it the {\em rotation group} of $\calP$. The automorphisms $\s_i$ are called {\em standard generators} of $\Gamma^+(\calP)$. If $\calP$ has type $\{p_1,\dots, p_{n-1}\}$ then the order of $\s_i$ is $p_i$ and therefore $\Gamma^+(\calP)$ is a suitable quotient of the even subgroup $[p_1,\dots,p_{n-1}]^+$ of the Coxeter group $[p_1,\dots,p_{n-1}]$ (see for example \cite[Chapter 3]{arp}).

If $\calP$ is chiral then $\Gamma(\calP) = \Gamma^+(\calP)$.
Whenever $\calP$ is regular, $\Gamma^+(\calP)$ has index at most $2$ in $\calP$; if the index is $2$ we say that $\calP$ is {\em orientably regular}, and it is {\em non-orientably regular} if $\Gamma(\calP)=\Gamma^+(\calP)$. In any of these cases, if $F$ is an $i$-face and $G$ is a $j$-face such that $F \le G$ and their ranks differ in at least $3$ then $\Gamma^+([F,G]) = \langle \s_{i+2}, \dots, \s_{j-1} \rangle$.

For a rotary polytope $\calP$, the standard generators of $\Gamma^+(\calP)$ satisfy
\begin{equation}\label{eq:1stRelation}
(\s_i \dots \s_j)^2 = \id \quad \mbox{for every $1\le i < j \le n-1$,}
\end{equation}
as well as the intersection condition
		\begin{equation} \label{eq:int-cond}
		A_I \cap A_J = A_{I \cup J} \quad \mbox{for every $I, J \subseteq \{0,\dots, n-1\}$},
		\end{equation}
where for $I \subseteq \{0,\dots,n-1\}$ the set $A_I$ denotes the stabilizer in $\Gamma^+(\calP)$ of those faces $F_i$ of the base flag with ranks $i\in I$. If $I = \{1,\dots, n-1\} \setminus \{i,i+1,\dots,j\}$ with $i<j$ then $A_I = \langle \sigma_{i+1}, \dots, \sigma_j \rangle$, which allows us to state the following lemma. For other sets $I$ the generating sets $\mathcal{X}_I$ of these stabilizers are more complicated (see \cite[Section 3]{chiral}).

\begin{lemma}\label{l:IntAi}
Let $\calP$ be a rotary polytope with $\Gamma^+(\calP) = \langle \s_1, \dots, \s_{n-1} \rangle$. If $j \le i+1 \le k$ then
\begin{equation}\label{eq:IntCond2}
\langle \s_1, \dots, \s_i \rangle \cap \langle \s_j,\dots, \s_{k} \rangle = \langle \s_j, \dots, \s_i \rangle.
\end{equation}
\end{lemma}

If $\calP$ is chiral we may choose the base flag in one or in the other flag orbit. These two choices produce non-equivalent sets of standard generators $\s_i$, in the sense that the defining relations for $\Gamma^+(\calP)$ will not be the same for the two sets. One may think of these two ways of looking at $\calP$ as a {\em left} and {\em right form} of the same object; we can go from one to the other just by `reflecting' our setting from the base flag into any of its adjacent flags. When doing this, we may take $\{\s_1^{-1},\s_1^2 \s_2, \s_3,\s_4,\dots,\s_{n-1}\}$ as the new set of standard generators for $\Gamma(\calP)$. For a chiral polyhedron, another convenient new set of generators is $\{ \s_1^{-1}, \s_2^{-1} \}$. The \emph{enantiomorph} of a chiral polytope $\calP$ (with an implicit base flag chosen) consists of the same polytope but where we change the base flag to any of its adjacent flags. We denote the enantiomorph of $\calP$ by $\calP^*$.
For more details about these forms see \cite{MonsonChiral}.

We mentioned that the rotation group of a rotary polytope is a group with a generating set satisfying (\ref{eq:1stRelation}) and the intersection condition (\ref{eq:int-cond}). Conversely, a group with a generating set satisfying (\ref{eq:1stRelation}) and a suitable version of (\ref{eq:int-cond}) is the rotation group of an orientable rotary polytope (that is, orientably regular or chiral).

The construction of the polytope from a group $\Gamma = \langle \s_1, \dots, \s_{n-1} \rangle$ is detailed in \cite[Section 5]{chiral}. It defines the $i$-face of the base flag as the subgroup of $\Gamma$ generated by the elements $\mathcal{X}_{\{i\}}$ of $A_{\{i\}}$ mentioned before Lemma \ref{l:IntAi}. The remaining $i$-faces are the cosets of the base $i$-face under the right action of $\Gamma$. It also establishes that two faces are incident if they have non-empty intersection. In particular, the sets of facets may be identified with the right cosets of $\langle \s_1,\dots,\s_{n-2} \rangle$ under $\Gamma$. Note that this construction can be performed even if the group does not satisfy the intersection condition. The output will still have well-defined flags and it is possible to talk about regularity through the action of its automorphism group.

If $\calP$ is non-orientably regular then that construction will produce the orientable double cover of $\calP$. It follows that there is a one-to-one correspondence between orientable rotary polytopes and groups satisfying (\ref{eq:1stRelation}) together with some version of (\ref{eq:int-cond}). For our purposes we find convenient the following version of (\ref{eq:int-cond}) that can be easily deduced from \cite[Lemma 10]{chiral}.

\begin{lemma}\label{l:eqIntCond}
Let $\Gamma = \langle \s_1,\dots,\s_{n-1} \rangle$ be a group where each $\s_i$ is nontrivial and the order of $\s_i \dots \s_j$ 
is $2$, for every $1\le i < j \le n-1$. Then $\Gamma$ satisfies the intersection condition (\ref{eq:int-cond}) if and only if
\begin{equation}\label{eq:eqIntCond}
\langle \s_1, \ldots, \s_i \rangle \cap \langle \s_j, \ldots, \s_{i+1} \rangle = \langle \s_j, \ldots, \s_i \rangle,
\end{equation} 
for every $2 \le j \le i+1 \le n-1$,
where if $j = i+1$ then we interpret the right-hand side as being the trivial group.
\end{lemma}

If $\calP$ is orientably regular (resp. chiral) with $\Gamma^+(\calP)= \langle \s_1,\dots,\s_{n-1} \rangle$ then $\calP^\delta$ is also orientably regular (resp. chiral) and, with respect to some flag, the $i$-th standard generator of $\Gamma^+(\calP^\delta)$ is $\s_{n-1-i}^{-1}$, for $i \in \{1,\dots,n-1\}$.

In upcoming sections we will be interested in normal subgroups contained in $\langle \s_i \rangle$ for some $i$. In those situations the following result will prove useful.

\begin{lemma}\label{l:rels1s3}
Let $\calP$ be a rotary $4$-polytope, and let $\Gamma^+(\calP) = \langle \sigma_1, \sigma_2, \sigma_{3} \rangle$.
\begin{enumerate}
\item For every $k$, $\s_3 \s_1^k \s_3^{-1} = \s_2^{-1} \s_1^{-k} \s_2.$
\item If $K$ is a subgroup of $\langle \s_1 \rangle$, then $\sigma_2^{-1} K \sigma_2 = K$ if and only if $\sigma_3^{-1} K \sigma_3 = K$.
\end{enumerate}
\end{lemma}

\begin{proof}
	We start with
	\begin{equation}\label{eq:magicalRelation} \s_3 \s_1 = (\s_1 \s_2)^2 \s_3 \s_1 (\s_2 \s_3)^2 = \s_1 \s_2 (\s_1 \s_2 \s_3)^2 \s_2 \s_3 = \s_1 \s_2^2 \s_3. \end{equation}
	It follows that
	\[ \s_3 \s_1^k = (\s_1 \s_2^2)^k \s_3. \]
	Then
	\[ \s_3 \s_1^k \s_3^{-1} = (\s_1 \s_2^2)^k = (\s_2^{-1} \s_1^{-1} \s_2)^k = \s_2^{-1} \s_1^{-k} \s_2. \]
	That proves part (a). Part (b) follows since $K = \langle \sigma_1^k \rangle$ for some $k$.
\end{proof}

\subsection{Covers and quotients}


If $\calP$ and $\calQ$ are orientable rotary $n$-polytopes such that $\calP$ covers $\calQ$ then there exists $N \triangleleft \Gamma^+(\calP)$ such that $\calQ \cong \calP/N$. In other words, the faces of $\calQ$ can be taken as the orbits of faces of $\calP$ under the action of $N$, and two of them are incident whenever an element in the orbit of one face is incident to some element in the orbit of the other face. 

Conversely, given $N \triangleleft \Gamma^+(\calP)$, the quotient $\calP/N$ is a polytope if and only if $\Gamma^+(\calP/N)$ satisfies (\ref{eq:1stRelation}) and the intersection condition (\ref{eq:eqIntCond}) with respect to the generators $\{\s_i N\}_{i \in \{1,\dots, n-1\}}$.

Whenever $\calP$ is chiral there exists a normal subgroup $X(\calP)$ of $\Gamma(\calP)$ satisfying that $\calP/X(\calP)$ is a regular structure (in the sense that all flags belong to the same orbit under $\Gamma(\calP/X(\calP)$), and that if $N \triangleleft \Gamma(\calP)$ is such that $\calP/N$ is a regular structure then $N \ge X(\calP)$. The group $X(\calP)$ is called the {\em chirality group} of $\calP$. Note that $\calP$ is regular if and only if $X(\calP)$ is trivial.

Elsewhere the chirality group has been introduced in other terms (see for example \cite{chirality-maps}, \cite{chiral-mix} and \cite{mix-ch}), but for our purposes the universal property of the chirality group mentioned here is more convenient.

The {\em mix} of two polytopes $\calP$ and $\calQ$ with base flags $\Phi_{\calP}$ and $\Phi_{\calQ}$, respectively, is the smallest structure $\calP \lozenge \calQ$ (which itself may or may not be a polytope) with well-defined ranks and adjacencies that covers simultaneously $\calP$ and $\calQ$, while mapping the base flag of $\calP \lozenge \calQ$ to $\Phi_{\calP}$ and $\Phi_{\calQ}$, respectively. As noted in \cite[Section 3]{k-orbit}, the choice of base flags may be relevant when performing the mix of two chiral polytopes. This is often taken into account by choosing a base flag from which to construct the standard generators of the automorphism group.

If $\calP$ and $\calQ$ are orientable rotary polytopes with $\Gamma^+(\calP) = \langle \s_1, \dots, \s_{n-1} \rangle$ and $\Gamma^+(\calQ) = \langle \s_1', \dots, \s_{n-1}' \rangle$ then $\Gamma^+(\calP \lozenge \calQ) = \langle \tau_1, \dots, \tau_{n-1} \rangle \le \Gamma^+(\calP) \times \Gamma^+(\calQ)$, where $\tau_i = (\sigma_i,\sigma_i')$. For convenience we also denote $\Gamma^+(\calP \lozenge \calQ)$ by $\Gamma^+(\calP) \lozenge \Gamma^+(\calQ)$.

The mix of two orientably regular polytopes is orientably regular. However, the mix of an orientable rotary polytope with a chiral polytope may be either orientably regular or chiral.

The next lemma relates the notions of quotient and mix of orientable rotary polytopes.

		\begin{lemma}\label{l:mixquotient}
		Let $\calP$ be an orientable rotary polytope with base flag $\Phi_0$ and let $K, N$ be normal subgroups of $\Gamma^+(\calP)$. Then
		\[\calP/(K\cap N) \cong (\calP / K) \lozenge (\calP / N),\]
		where the base flags of $\calP/K$ and $\calP/N$ are taken as $\Phi_0 \cdot K$ and $\Phi_0 \cdot N$, respectively.
		\end{lemma}
		
\begin{proof}
The polytope $\calP/(K\cap N)$ covers $\calP/K$ mapping a face $F \cdot (K \cap N)$ to the face $F \cdot K$. Similarly, it covers $\calP/N$. Hence $\calP/(K\cap N)$ covers the mix $(\calP / K) \lozenge (\calP / N)$.

Let $\Gamma^+(\calP)=\langle \s_1, \dots, \s_{n-1} \rangle$. Then there is a group epimorphism from $\Gamma^+(\calP/(K \cap N))$ to $\Gamma^+((\calP / K) \lozenge (\calP / N))$ mapping $\s_i \cdot (K \cap N)$ to $(\s_i \cdot K, \s_i \cdot N)$ for $i \in \{1,\dots, n-1\}$. This epimorphism sends the element $\s_{i_1} \cdots \s_{i_k} \cdot (K \cap N)$ to $(\s_{i_1} \cdots \s_{i_k} \cdot K, \s_{i_1} \cdots \s_{i_k} \cdot N)$. The latter is trivial if and only if $\s_{i_1} \cdots \s_{i_k} \in K \cap N$. Since the kernel of the epimorphism is trivial, the isomorphism holds.
\end{proof}

Given a chiral polytope $\calP$ there exists a smallest regular structure $\calR$ with well-defined ranks and adjacencies of flags that covers $\calP$ (even if this structure is not a polytope itself), in the sense that every regular polytope that covers $\calP$ also covers $\calR$. We shall call this structure the {\em smallest regular cover} of $\calP$.

Sometimes the smallest regular cover of $\calP$ is a polytope itself; for example, when the facets or the vertex-figures are regular (see \cite[Corollary 7.5]{mixing-and-monodromy}). If the smallest regular cover of $\calP$ is a polytope then it is elsewhere also called the {\em minimal regular cover} of $\calP$; otherwise, $\calP$ may have multiple polytopal regular covers that are minimal in the partial order given by the covering relation.

The smallest regular cover $\calR$ of a chiral polytope $\calP$ is the regular structure constructed (in the sense of \cite{chiral}) from the group $\Gamma(\calP) \lozenge \Gamma(\calP^*)$, where $\calP^*$ is the enantiomorph of $\calP$ (see \cite[Section 7]{mixing-and-monodromy}). We may assume that if $\Gamma(\calP) = \langle \s_1, \dots, \s_{n-1} \rangle$ then $\Gamma^+(\calR) = \langle (\s_1,\s_1^{-1}), (\s_2,\s_1^2 \s_2), (\s_3,\s_3), \ldots, (\s_{n-1},\s_{n-1})\rangle$.

We next relate the chirality group of a chiral polytope with its smallest regular cover. This is a direct consequence of \cite[Remark 7.3]{mixing-and-monodromy}.

\begin{lemma}\label{l:ChirGruIsKer}
Let $\calP$ be a chiral polytope and $\calR$ its smallest regular cover. Then $X(\calP)$ is isomorphic to the kernel of the quotient from $\Gamma^+(\calR)$ to $\Gamma(\calP)$. 
\end{lemma}

The following result relates the smallest regular covers of chiral polytopes with that of one of its facets.

\begin{lemma}\label{l:mrcFacets}
Let $\calP$ be a chiral polytope with chiral facets isomorphic to $\calQ$. Then the facets of the smallest regular cover of $\calP$ are isomorphic to the smallest regular cover of $\calQ$.
\end{lemma}

\begin{proof}
Since the facets of $\calP$ are chiral, $\calP$ has rank $n \ge 4$.

Let $\Gamma(\calP)= \langle \sigma_1, \dots, \sigma_{n-1} \rangle$, let $\calR_{\calP}$ be the smallest regular cover of $\calP$, and let $\calR_{\calQ}$ be the smallest regular cover of $\calQ$. Then $\Gamma^+(\calR_{\calP})=\langle \sigma_1, \dots, \sigma_{n-1} \rangle \lozenge \langle \sigma_1^{-1}, \sigma_1^2 \sigma_2, \sigma_3, \dots, \sigma_{n-1}\rangle$ and $\Gamma^+(\calR_{\calQ}) = \langle \sigma_1, \dots, \sigma_{n-2} \rangle \lozenge \langle \sigma_1^{-1}, \sigma_1^2 \sigma_2, \sigma_3, \dots, \sigma_{n-2}\rangle$. Since the orientation preserving automorphism group of the facet of $\calR_P$ is $\Gamma^+(\calR_{\calQ})$, the lemma holds.
\end{proof}

We conclude this section with a result that relates the chirality group of a chiral polytope $\calP$ with that of its facets.

		\begin{lemma}\label{l:chirGroupFacet}
		Let $\calP$ be a chiral polytope with chiral facets isomorphic to $\calQ$. Then $X(\calQ) \le X(\calP)$.
		\end{lemma}

		\begin{proof}
		Let $\calR_{\calP}$ be the smallest regular cover of $\calP$, and let $\calR_{\calQ}$ be the smallest regular
		cover of $\calQ$. Then, by Lemma \ref{l:mrcFacets}, the facets of $\calR_{\calP}$ are isomorphic to $\calR_{\calQ}$.
		By Lemma \ref{l:ChirGruIsKer}, $X(\calQ)$ is the kernel of the natural covering $\eta_{\calQ}$ from $\G^+(\calR_{\calQ})$ to $\G^+(\calQ)$,
		whereas $X(\calP)$ is the kernel of the natural covering $\eta_{\calP}$ from $\G^+(\calR_{\calP})$ to $\G^+(\calP)$.
		Since the kernel of $\eta_{\calQ}$ is contained in the kernel of $\eta_{\calP}$,
		the result follows.
		\end{proof}

	\subsection{Tight polytopes}

		A polytope of type $\{p_1, p_2, \ldots, p_{n-1}\}$ has at least $2p_1 p_2 \cdots p_{n-1}$ flags,
		and if it has exactly that many flags, we say it is \emph{tight} \cite[Prop. 3.3]{tight-polytopes}.

The first mention of the property of tightness occured in \cite{smallest-regular}, while searching for the smallest regular polytopes of each rank. There it was proven that for $n \ge 4$, the regular $n$-polytopes with fewest flags are always tight. Their study was extended in \cite{tight-polytopes} to equivelar polytopes that may not be regular. In particular, it was proven there that an equivelar polytope is tight if and only if every section of rank $3$ is flat. It follows that every section of a tight polytope is itself tight. The following lemma is a natural consequence of this fact.

\begin{lemma}\label{l:flatflat}
Let $\calP$ and $\calQ$ be tight rotary polytopes with types $\{p,q\}$ and $\{q,r\}$, respectively. Suppose that $\Gamma(\calP)= [p,q]^+ /N_1$ and $\Gamma(\calQ) = [q,r]^+ /N_2$ where $N_1 \triangleleft [p,q]^+$ and $N_2 \triangleleft [q,r]^+$ are subgroups induced by the sets of relations $R_1$ and $R_2$, respectively. Then a rotary $4$-polytope with facets isomorphic to $\calP$ and vertex-figures isomorphic to $\calQ$ exists if and only if the group $[p,q,r]^+/ N_3$ has order $pqr$ and satisfies the intersection condition (\ref{eq:eqIntCond}), where $N_3$ is the subgroup induced by the relations in $R_1$ in the first two generators and the relations $R_2$ in the last two generators. Moreover, such a $4$-polytope must be unique.
\end{lemma}

Tight regular and chiral polyhedra were studied more deeply in \cite{tight2}, \cite{tight-chiral-polyhedra} and \cite{tight3}. We summarize relevant results on these polyhedra in Section \ref{s:polyhedra}. Some results on regular polytopes of higher ranks can be found in \cite{tight2}.

 The next proposition summarizes Corollary 3.4 and Theorem 3.5 of \cite{non-flat}.

\begin{proposition}\label{prop:facvfChiral}
 \begin{itemize}
  \item[(a)] If $\calP$ is a tight chiral $4$-polytope then it has chiral facets or chiral vertex-figures (or both).
  \item[(b)] If $\calP$ is a tight chiral $5$-polytope then it has chiral facets, vertex-figures, and medial sections.
  \item[(c)] There are no tight chiral $n$-polytopes for $n \ge 6$.
 \end{itemize}
\end{proposition}

		Since we shall work with the automorphism groups of chiral polytopes in place of the polytopes themselves, it
		is useful to have a characterization of tightness that is entirely group-theoretic.

		\begin{proposition} \label{prop:tight-group-factoring}
		Suppose that $\calP$ is an orientable rotary $n$-polytope of type $\{p_1, \ldots, p_{n-1}\}$, 
		with $\G^+(\calP) = \langle \s_1, \ldots, \s_{n-1} \rangle$. Then the following are equivalent:
		\begin{enumerate}
		\item $\calP$ is tight.
		\item $|\G^+(\calP)| = p_1 \cdots p_{n-1}$.
		\item $\G^+(\calP) = \langle \s_1 \rangle \cdots \langle \s_{n-1} \rangle$.
		\end{enumerate}
		\end{proposition}
		
		\begin{proof}
		The equivalence of (a) and (b) follows from the fact that $|\G^+(\calP)|$ is equal to half the number of flags.
		
		Next we show that (b) and (c) are equivalent. For each $1 \leq i \leq n-1$, let
		\[ S_i = \langle \s_i \rangle \cdots \langle \s_{n-1} \rangle. \]
		Then $|S_{n-1}| = p_{n-1}$, and for $i < n-1$,
		\[ S_i = \langle \s_i \rangle S_{i+1}. \]
		Therefore,
		\[ |S_i| = \frac{|\langle \s_i \rangle| \cdot |S_{i+1}|}{|\langle \s_i \rangle \cap S_{i+1}|}, \]
		and since $\G^+(\calP)$ satisfies the intersection condition (\ref{eq:eqIntCond}), the intersection on bottom
		is trivial, and so
		\[ |S_i| = p_i \cdot |S_{i+1}|. \]
		It follows that $|S_1| = p_1 \cdots p_{n-1}$. This shows that (c) implies (b).

Conversely,
		if $|\G^+(\calP)| = p_1 \cdots p_{n-1}$, then $\G^+(\calP)$ has the same order as its subset $S_1$,
		which implies that $\G^+(\calP) = S_1$.
		\end{proof}

		Note that (b) and (c) are equivalent only in the presence of the intersection condition.

		In light of Proposition~\ref{prop:tight-group-factoring}, we will say that the group $\G = \langle \s_1, \ldots, \s_{n-1} \rangle$
		is tight provided that $\G = \langle \s_1 \rangle \cdots \langle \s_{n-1} \rangle$. Then $\G$ is the rotation
		group of a tight orientable rotary polytope if and only if $\G$ is tight and it
		satisfies the intersection condition (\ref{eq:eqIntCond}). The following result is immediate:
		
		\begin{proposition} \label{prop:tight-gp-quo}
		If $\G$ is tight, then any quotient of $\G$ is tight. If $\calP$ is a tight orientable rotary polytope then any quotient of $\calP$ is tight.
		\end{proposition}

Proposition \ref{prop:tight-gp-quo} imposes a restriction on the quotients of tight orientable rotary polytopes. The contrapositive of the next proposition imposes another restriction to quotients of tight orientably regular polytopes, namely that tight regular polytopes do not have chiral quotients.
	
\begin{proposition}\label{prop:RegularsNotCoverChirals}
If $\calP$ is a tight orientable rotary $n$-polytope that covers a chiral $n$-polytope then $\calP$ itself is chiral.
\end{proposition}

\begin{proof}
Let $\calQ$ be a chiral quotient of $\calP$. We proceed by induction over $n$. By Proposition \ref{prop:facvfChiral} (c), it is only necessary to show the statement for $n \in \{3,4,5\}$.

The case when $n=3$ was proven in \cite[Prop. 2.5]{tight-chiral-polyhedra}. If $n \in {4,5}$, then by Proposition \ref{prop:facvfChiral} either the facets or the vertex-figures of $\calQ$ are chiral $(n-1)$-polytopes. Since the facets and vertex-figures of $\calQ$ are quotients of the facets and vertex-figures of $\calP$, the inductive hypothesis implies that the facets or vertex-figures of $\calP$ must be chiral. Hence $\calP$ is chiral.
\end{proof}

Propositions \ref{prop:tight-gp-quo} and \ref{prop:RegularsNotCoverChirals} have the following consequence. When taking polytopal quotients of a tight chiral polytope $\calP$ by normal subgroups of $\Gamma^+(\calP)$, we obtain tight orientably regular or chiral polytopes, and if $\calP$ is orientably regular then the quotients are tight and regular. This suggests to try to find successive proper quotients of tight chiral polytopes until we obtain tight regular polytopes. As we shall see, this is always possible. Proposition \ref{prop:NonTrivialCore} gives a condition for such quotients to exist. Other conditions will be given in Sections \ref{s:ChiralChiral} and \ref{s:ChiralRegular}.

The chiral polytopes we will be interested in typically have a cyclic chirality group, generated by a power of some $\s_i$.
The following result describes circumstances where this property is preserved when taking quotients.

\begin{lemma}\label{l:ChirGroupS2}
Let $\calP$ be a tight chiral polytope with $\Gamma(\calP) = \langle \sigma_1, \ldots, \s_n \rangle$ and $\calQ$ a chiral quotient of $\calP$ with $\Gamma(\calQ) = \langle \sigma_1', \ldots, \sigma_n' \rangle$. If $X(\calP) \le \langle \s_2 \rangle$ then $X(\calQ) \le \langle \s_2' \rangle$.
\end{lemma}

\begin{proof}
Let $K \triangleleft \Gamma(\calP)$ such that $\calQ = \calP/K$, and let $\calR = \calP/K X(P)$. Then $\calR$ is a quotient of
$\calP / X(\calP)$, and since the latter is regular, Propositions~\ref{prop:tight-gp-quo} and \ref{prop:RegularsNotCoverChirals}
imply that $\calR$ is regular as well. Now, $\calR$ is the quotient of $\calQ$ by $KX(P)/K$, and since $\calR$ is regular, that
implies that $X(\calQ)$ is contained in $KX(\calP)/K$, which is the image of $X(\calP)$ in $\G(\calQ)$, and thus contained
in $\langle \s_2' \rangle$.
\end{proof}

Next, we describe useful structural properties of the normal subgroups of the rotation group of tight orientable rotary polytopes.

		\begin{lemma}\label{l:subgroups}
		Let $\calP$ be a tight orientable rotary $n$-polytope with $\Gamma^+(\calP) = \langle \sigma_1, \dots, \sigma_{n-1} \rangle$ and let 
		$K \triangleleft \Gamma^+(\calP)$ such that $\calP/K$ is a tight orientable rotary $n$-polytope. 
		Then there exist non-negative integers $\alpha_1, \dots, \alpha_{n-1}$ such that
		\[K = \langle \sigma_1^{\alpha_1} \rangle \langle \sigma_2^{\alpha_2} \rangle \cdots \langle \sigma_{n-1}^{\alpha_{n-1}} \rangle.\]
		Moreover, $\calP/K$ has type $\{\alpha_1,\dots,\alpha_{n-1}\}$.
		\end{lemma}

		\begin{proof}
		For $1 \leq i \leq n-1$, let $\alpha_i$ be the smallest positive integer such that $\s_i^{\alpha_i} \in K$,
		and let $H = \langle \s_1^{\alpha_1} \rangle \cdots \langle \s_{n-1}^{a_{n-1}} \rangle$. Then clearly
		$H \subseteq K$. To show the reverse inclusion, let $\gamma \in K$. By Proposition~\ref{prop:tight-group-factoring},
		we may write $\gamma$ as $\s_1^{\beta_1} \cdots \s_{n-1}^{\beta_{n-1}}$ for some exponents $\beta_i$. 
		Since $\gamma \in K$, we have that for every $i$, 
		\[ K \s_1^{\beta_1} \cdots \s_i^{\beta_i} = K (\s_{i+1}^{\beta_{i+1}} \cdots \s_{n-1}^{\beta_{n-1}})^{-1}. \]
		Then, writing $\ol{\s_i}$ for the image of $\s_i$ in $\G^+(\calP)/K$, we get that
		\[ \ol{\s_1^{\beta_1} \cdots \s_i^{\beta_i}} = \ol{(\s_{i+1}^{\beta_{i+1}} \cdots \s_{n-1}^{\beta_{n-1}})^{-1}}. \]
		Since $\G^+(\calP)/K$ is the rotation group of a rotary polytope,
		Equation~(\ref{eq:IntCond2}) implies that $\ol{\s_1^{\beta_1} \cdots \s_i^{\beta_i}} = 1$, which means that
		$\s_1^{\beta_1} \cdots \s_i^{\beta_i} \in K$ for every $i$. In particular, $\s_1^{\beta_1} \in K$,
		from which it follows that $\s_2^{\beta_2} \in K$ (since $\s_1^{\beta_1} \s_2^{\beta_2} \in K$),
		and continuing in this way it follows that each $\s_i^{\beta_i} \in K$. By our choice of exponents $\alpha_i$, that means
		that each $\beta_i$ is divisible by $\alpha_i$, and so $\gamma \in H$.
		
		The type of $\calP/K$ follows from Proposition \ref{prop:tight-gp-quo} and the fact that $K$ has order $p_1 \cdots p_{n-1} / \alpha_1 \cdots \alpha_{n-1}$.
		\end{proof}

		\begin{proposition} \label{prop:vfig-quo}
		Suppose that $\calP$ is a tight orientable rotary $n$-polytope with $\G^+(\calP) = \langle \s_1, \ldots, \s_{n-1} 
		\rangle$, and let $N = \langle \s_1^{a_1} \rangle \langle \s_2^{a_2} \rangle \cdots \langle \s_{n-1}^{a_{n-1}} \rangle$ be a normal subgroup of $\G^+(\calP)$. If $N$ does not contain any generator $\s_i$, then $\G^+(\calP) / N$ 
		is the rotation group of a tight orientable rotary polytope.
		\end{proposition}

		\begin{proof}
		Let $\G^+(\calP) / N = \langle \ol{\s_1}, \ldots, \ol{\s_{n-1}} \rangle$. 
		Since no generator $\s_i$ is in $N$, it follows that each $\ol{\s_i}$ has order at least 2. Then to prove that
		$\G^+(\calP) / N$ is the rotation group of an orientable rotary polytope, by Lemma \ref{l:eqIntCond} it suffices to show that
		\[ \langle \ol{\s_1}, \ldots, \ol{\s_i} \rangle \cap \langle \ol{\s_j}, \ldots, \ol{\s_{i+1}} \rangle = \langle \ol{\s_j}, \ldots, \ol{\s_i} \rangle, \]
for all $i$ and $j$ such that $2 \leq j \leq i+1 \leq n-1$. (In fact, it suffices to show
that the subgroup on the left is included in the subgroup on the right, since the reverse inclusion is obvious.) Tightness will then follow 
		from Proposition~\ref{prop:tight-gp-quo}.

		Consider an element of $\G^+(\calP) / N$ that lies in 
		\[ \langle \ol{\s_1}, \ldots, \ol{\s_i} \rangle \cap \langle \ol{\s_j}, \ldots, \ol{\s_{i+1}} \rangle. \]
		We may write this element as $\ol{\varphi_1} = \ol{\varphi_2}$, where
		\[ \varphi_1 \in \langle \s_1, \ldots, \s_i \rangle \]
		and
		\[ \varphi_2 \in \langle \s_j, \ldots, \s_{i+1} \rangle. \]
		Then $\varphi_1 = \gamma \varphi_2$ for some $\gamma \in N$.
		Since $\calP$ is tight, Proposition~\ref{prop:tight-group-factoring}(c) says that we may write $\gamma = \s_1^{b_1} \cdots \s_{n-1}^{b_{n-1}}$.
		Setting $\gamma_1 = \s_1^{b_1} \cdots \s_{j-1}^{b_{j-1}}$ and $\gamma_2 = \s_j^{b_j} \cdots \s_{n-1}^{b_{n-1}}$,
		we have that by definition $\gamma_1$ and $\gamma_2$ both lie in $N$. Now,
		\[ \gamma_1^{-1} \varphi_1 = \gamma_2 \varphi_2, \]
		and it follows that
		\[ \gamma_1^{-1} \varphi_1 \in \langle \s_1, \ldots, \s_i \rangle \cap \langle \s_j, \ldots, \s_{n-1} \rangle. \]
Then since $\G^+(\calP)$ satisfies the intersection condition, it follows from
		Lemma \ref{l:IntAi} that $\gamma_1^{-1} \varphi_1 \in \langle \s_j, \ldots, \s_i \rangle$. And since $\gamma_1 \in N$,
		this implies that $\ol{\varphi_1} \in \langle \ol{\s_j}, \ldots, \ol{\s_i} \rangle$, which is what we wanted to show.
		\end{proof}

When considering $\Gamma^+(\calP)$ as a group acting on the set of $i$-faces of $\calP$ for some $i$, the kernel of this action is a natural normal subgroup of $\calP$ to consider. (Recall that the {\em kernel} of the action of a group $\Gamma$ on a set $X$ is the subgroup of $\Gamma$ fixing $X$ pointwise.) The next results give sufficient conditions for the kernel of the action on the vertex set to be non-trivial.

\begin{lemma}\label{l:FixOneFixAll}
Let $\calP$ be a tight orientable rotary polyhedron. If $\gamma\in \Gamma^+(\calP)$ fixes a vertex and one of its neighbors then it fixes all vertices of $\calP$.
\end{lemma}

\begin{proof}
Let $u_0$ be the base vertex of $\calP$. Let $\G^+(\calP) = \langle \s_1, \s_2 \rangle$, and 
let $\gamma \in \Gamma^+(\calP)$ such that it fixes $u_0$ and one of its neighbors $v_0$. 

Since the stabilizer of $u_0$ is $\langle \s_2 \rangle$ then $\gamma = \s_2^a$ for some $a$.
Now, if $\s_2^a$ fixes $v_0$ then it must fix all neighbors of $u_0$, since all of them are images of $v_0$ under $\langle \s_2 \rangle$. Since the choice of base vertex is arbitrary, we have proven that if $\gamma$ fixes a vertex $u$ and one of its neighbors then it fixes all neighbors of $u$.

The result then follows from the connectivity of the $1$-skeleton of $\calP$.
\end{proof}

The fact that the base facet of a tight polytope $\calP$ contains all vertices of $\calP$ implies the following corollary.

\begin{corollary}\label{c:FixOneFixAll}
Let $\calP$ be a tight orientable rotary $n$-polytope with $\Gamma^+(\calP) = \langle \s_1, \dots, \s_{n-1} \rangle$. If $\s_2^a$ fixes a neighbor of the base vertex then it fixes all vertices of $\calP$.
\end{corollary}

		\begin{corollary}\label{cor:NTkernel}
		Let $\calP$ be a tight orientable rotary $n$-polytope with type $\{p_1,\dots, p_{n-1}\}$ with $p_1 \le p_2$. Then the kernel of the action of $\Gamma^+(\calP)$ on the vertex set is non-trivial.
		\end{corollary}
		
		\begin{proof}
		If $\calP$ is a tight polytope of type $\{p_1, \ldots, p_{n-1}\}$, then it has $p_1$ vertices. The automorphism $\s_2$
		fixes the base vertex while permuting the remaining $p_1 - 1$. If $p_1 \leq p_2$, then each neighbor of the base vertex
		must have a nontrivial stabilizer under $\langle \s_2 \rangle$, since the group has order $p_2$, which is larger
		than the largest possible orbit.
		\end{proof}
		
Now we are ready to exhibit a proper normal subgroup $N$ of $\Gamma^+(\calP)$ that is a key element in discussions in Sections \ref{s:ChiralChiral} and \ref{s:ChiralRegular}.

\begin{proposition}\label{prop:NonTrivialCore}
Let $\calP$ be a tight orientable rotary $n$-polytope with $n \ge 3$ with type $\{p_1, \dots, p_{n-1}\}$ satisfying that $p_1 \ge p_2$, and rotation group $\Gamma^+(\calP)= \langle \sigma_1, \dots, \sigma_{n-1} \rangle$. Then there exists an integer $k$ such that $\langle \sigma_1^k \rangle$ is a non-trivial normal subgroup of $\Gamma^+(\calP)$.
\end{proposition}

\begin{proof}
By the dual version of Corollary \ref{cor:NTkernel} the group $\langle \sigma_1, \sigma_2 \rangle$ has a non-trivial kernel when acting on the $2$-faces of the base $3$-face of $\cal P$. These $2$-faces correspond to cosets of $\langle \sigma_1 \rangle$ in $\langle \sigma_1, \sigma_2 \rangle$. Then there exists $k \in \{1,\dots, p_1-1\}$ such that
$\langle \sigma_1 \rangle \sigma_2^\ell \sigma_1^k = \langle \sigma_1 \rangle \sigma_2^\ell$ for every $\ell$. In particular, when $\ell=-1$ this implies that $\sigma_2^{-1} \sigma_1^k \sigma_2 \in \langle \sigma_1 \rangle$. Since the latter group is cyclic, we have that $\langle \sigma_1^k \rangle$ is normal in $\langle \sigma_1, \sigma_2 \rangle$. The result follows from Lemma \ref{l:rels1s3} and commutativity of $\s_1^k$ with $\s_i$ for every $i \ge 4$.
\end{proof}

\section{Tight orientable rotary polyhedra and $4$-polytopes}\label{s:polyhedra}

Much of the discussion on tight chiral $n$-polytopes for $n \ge 4$ in Sections \ref{s:ChiralChiral}, \ref{s:ChiralRegular} and \ref{s:5polytopes} is based on what we know about tight orientable rotary polyhedra. In this section we summarize some important facts about them.

We start with a simple result related to Lemma \ref{l:NoChiral2}, and one of its consequences for tight orientable rotary polyhedra.

\begin{lemma}\label{l:SchType2}
For every $p \ge 2$ there is a unique polyhedron of type $\{p,2\}$ and it is regular.
\end{lemma}

\begin{proof}
Let $\calP$ be a polyhedron with type $\{p,2\}$. Then every vertex of $\calP$ is incident with precisely two edges and precisely two facets. Since adjacent vertices belong to the same two facets, the connectivity of $\calP$ forces $\calP$ itself to have only two facets. It follows that $\calP$ is isomorphic to the face-lattice of the map on the sphere whose $1$-skeleton is an equatorial $p$-gon and its two facets are the northern and southern hemispheres. Clearly $\calP$ is regular.
\end{proof}

\begin{lemma}\label{l:NormalThen2}
If $\calP$ is an orientable rotary tight polyhedron with $\Gamma^+(\calP)=\langle \s_1, \s_2 \rangle$ and $\langle \s_1 \rangle \triangleleft \Gamma^+(\calP)$, then $\calP$ has type $\{p,2\}$ for some $p$. In particular, $\calP$ is regular.
\end{lemma}

\begin{proof}
If $\s_2^{-1} \langle \sigma_1 \rangle \s_2 = \langle \s_1 \rangle$ then $\s_2^{-1} \s_1 \s_2 = \s_1^k$ for some $k$. Now, $\s_2^{-1} \s_1 \s_2= \s_2^{-1} \s_2^{-1} \s_1^{-1}$, implying that $\s_2^{-2} = \s_1^{k+1}$. The intersection condition (\ref{eq:eqIntCond}) tells us that $\s_2$ has order $2$ and hence the type of $\calP$ is $\{p,2\}$ for some $p$. Lemma \ref{l:SchType2} implies the regularity of $\calP$.
\end{proof}

The rotation groups of tight orientable rotary polyhedra have many normal subgroups contained in the vertex or facet stabilizer. In the next result we describe some of these normal subgroups.

	\begin{proposition}
	\label{prop:normal-props}
	Suppose $\calP$ is a chiral or orientable rotary polyhedron of type $\{p, q\}$, with $\G^+(\calP) = \langle \s_1, \s_2 \rangle$.
	If $\langle \s_2^a \rangle \triangleleft \Gamma^+(\calP)$, then $\s_2^a \s_1 = \s_1 \s_2^{sa}$ for some $s$ such that $s^2 \equiv 1$ (mod $q/a$).
	In particular, $\s_1^2$ commutes with $\s_2^a$, and if $p$ is odd, then $\s_2^a$ is central.
	\end{proposition}
	
	\begin{proof}
	The subgroup $\langle \s_2^a \rangle$ is normal if and only if $\s_1^{-1} \s_2^a \s_1 = \s_2^{sa}$ for some $s$.
	Furthermore, we note that
	\begin{align*}
	\s_2^a &= (\s_1 \s_2)^{-2} \s_2^a (\s_1 \s_2)^2 \\
	&= (\s_2^{-1} \s_1^{-1} \s_2^{-1}) \s_2^{sa} (\s_2 \s_1 \s_2) \\
	&= \s_2^{-1} \s_1^{-1} \s_2^{sa} \s_1 \s_2 \\
	&= \s_2^{-1} \s_2^{s^2 a} \s_2 \\
	&= \s_2^{s^2 a},
	\end{align*}
	so that $a \equiv s^2 a$ (mod $q$), and thus $s^2 \equiv 1$ (mod $q/a$). It is now clear then that $\s_1^2$ commutes
	with $\s_2^a$, and if $p$ is odd, then $\langle \s_1^2 \rangle = \langle \s_1 \rangle$ so that $\s_1$ commutes
	with $\s_2^a$ as well.
	\end{proof}

According to \cite[Proposition 4.6]{tight3}, an orientably regular polyhedron has no multiple edges if and only if $\langle \s_2 \rangle$ is core-free. The following lemma is a consequence of \cite[Lemma 4.7 (c)]{tight3}.

\begin{lemma}\label{l:sigma2IsNormal}
Let $\calP$ be a tight regular polyhedron with $\Gamma^+(\calP)=\langle \s_1,\s_2 \rangle$ and $\langle \s_2 \rangle$ core-free. Then $\langle \s_1^2 \rangle \triangleleft \Gamma^+(\calP)$.
\end{lemma}

Tight orientably regular polyhedra with no multiple edges were classified in \cite[Theorem 4.13]{tight3}. The next theorem is a direct consequence.

\begin{theorem}\label{t:TightRegularHedra}
The types of the tight orientably regular polyhedra with no multiple edges are:
\begin{itemize}
 \item[(a)] $\{p,2\}$ for some $p \ge 2$,
 \item[(b)] $\{2q,q\}$ for some odd integer $q\ge 3$,
 \item[(c)] $\{p,q\}$ with $p = 2^{\alpha_1}P_2^{\alpha_2} \cdots P_k^{\alpha_k}$ for some $\alpha_1 > 0$, some distinct odd primes $P_2,\dots,P_k$, and $q$ a proper even divisor of $p$ satisfying that
 \begin{itemize}
  \item the maximal power of $2$ dividing $q$ is either $2$, $4$ or $2^{\alpha_1-1}$, and if it is $4$ then $\alpha_1 \geq 3$,
  \item for $i\in \{2,\dots,k\}$, either $P_i^{\alpha_i}$ divides $q$ or $P_i$ is coprime with $q$.
 \end{itemize}
\end{itemize}
\end{theorem}

In \cite{tight-chiral-polyhedra} an {\em atomic} chiral polyhedron was defined as a tight chiral polyhedron with type $\{p,q\}$ that covers no chiral polyhedron of type $\{p',q\}$ or of type $\{p,q'\}$ for $p'$ a proper divisor of $p$ and $q'$ a proper divisor of $q$. It is easy to see that every tight chiral polyhedron covers an atomic chiral polyhedron.

The atomic chiral polyhedra were classified in \cite[Lemma 4.10, Theorem 4.11, Theorem 4.14]{tight-chiral-polyhedra}. Here we summarize and slightly simplify this classification (see \cite[Theorem 4.15]{tight-chiral-polyhedra}).

	\begin{table}[htbp]
	\centering
	\begin{tabular}{|l|l|l|l|l|} \hline
	$\calP$ & Extra relations & $X(\calP)$ & $\enant{\calP}$ & Notes \\ \hline
	$\calP(2m, m^{\alpha})_k$ & $\s_2 \s_1^2 = \s_1^2 \s_2^{1+k m^{\alpha-1}}$ & $\langle \s_2^{m^{\alpha-1}} \rangle$ & $\calP(2m, m^{\alpha})_{m-k}$ & $m$ odd prime, $\alpha \geq 2$ \\
	& & & & $1 \leq k \leq m-1$ \\ \hline
	$\calP(m^{\alpha}, 2m)_k$ & $\s_2^2 \s_1= \s_1^{1+k m^{\alpha-1}} \s_2^2$ & $\langle \s_1^{m^{\alpha-1}} \rangle$ & $\calP(m^{\alpha}, 2m)_{m-k}$ & $m$ odd prime, $\alpha \geq 2$ \\
	& & & & $1 \leq k \leq m-1$ \\ \hline
	
	$\calP(8, 2^{\beta})_{\epsilon}$ & $\s_2 \s_1^2 = \s_1^2 \s_2^{1+\epsilon 2^{\beta-2}}$  & $\langle \s_2^{2^{\beta-1}} \rangle$ & $\calP(8, 2^{\beta})_{-\epsilon}$ & $\beta \geq 5$, $\epsilon = \pm 1$ \\ \hline
	$\calP(2^{\beta}, 8)_{\epsilon}$ & $\s_1 \s_2^2 = \s_2^2 \s_1^{1+\epsilon 2^{\beta-2}}$  & $\langle \s_1^{2^{\beta-1}} \rangle$ & $\calP(2^{\beta}, 8)_{-\epsilon}$ & $\beta \geq 5$, $\epsilon = \pm 1$ \\ \hline
	
	$\calP(2^{\beta-1}, 2^{\beta})_{\epsilon}$ & $\s_2^{-1} \s_1 = \s_1^{-1+2^{\beta-2}} \s_2^{-3+\epsilon 2^{\beta-2}}$ & $\langle \s_2^{2^{\beta-1}} \rangle$ & $\calP(2^{\beta-1}, 2^{\beta})_{-\epsilon}$ & $\beta \geq 5$, $\epsilon = \pm 1$ \\
	& $\s_2 \s_1^{-1} = \s_1^{1+2^{\beta-2}} \s_2^{3+\epsilon 2^{\beta-2}}$ & & & \\ \hline
	$\calP(2^{\beta}, 2^{\beta-1})_{\epsilon}$ & $\s_1^{-1} \s_2 = \s_2^{-1+2^{\beta-2}} \s_1^{-3+\epsilon 2^{\beta-2}}$ & $\langle \s_1^{2^{\beta-1}} \rangle$ & $\calP(2^{\beta}, 2^{\beta-1})_{-\epsilon}$ & $\beta \geq 5$, $\epsilon = \pm 1$ \\
	& $\s_1 \s_2^{-1} = \s_2^{1+2^{\beta-2}} \s_1^{3+\epsilon 2^{\beta-2}}$ & & & \\ \hline
	\end{tabular}
	\caption{The atomic chiral polyhedra. An atomic chiral polyhedron with name $\calP(p,q)_t$ has an automorphism group that is a quotient of $[p,q]^+$ by the relation(s) given in the ``Extra relations'' column,
	and the subscript indicates the name of an additional parameter.}
	\label{tab:atomic-polyhedra}
	\end{table}

\begin{theorem} \label{t:atomic-polyhedra}
Every atomic chiral polyhedron $\calP$ is one of the polyhedra in Table~\ref{tab:atomic-polyhedra}, with chirality group $X(\calP)$ and enantiomorph $\enant{\calP}$ as described in the table.
\end{theorem}

\begin{proof}
First we will prove the claim for atomic chiral polyhedra of type $\{2m, m^{\alpha}\}$ and $\{m^{\alpha}, 2m\}$.
We start by noting that for any rotation group $\langle \s_1, \s_2 \rangle$ and
for all $t$, the relation $\s_2^{-1} \s_1 = \s_1^3 \s_2^t$ is equivalent to $\s_2 \s_1^2 = \s_1^2 \s_2^t$, since:
\begin{align*}
\s_2^{-1} \s_1 &= \s_1^3 \s_2^t && \text{Multiply both sides by $\s_1^{-1}$ on the left} \\
\s_1^{-1} \s_2^{-1} \s_1 &= \s_1^2 \s_2^t && \text{Use $\s_1^{-1} \s_2^{-1} = \s_2 \s_1$} \\
\s_2 \s_1^2 &= \s_1^2 \s_2^t.
\end{align*}
Similarly, for all $t$, the relation $\s_2 \s_1^{-1} = \s_1^{-3} \s_2^t$ is equivalent to $\s_1^2 \s_2 = \s_2^{-t} \s_1^2$, since:
\begin{align*}
\s_2 \s_1^{-1} &= \s_1^{-3} \s_2^t && \text{Multiply both sides by $\s_1$ on the left } \\
\s_1 \s_2 \s_1^{-1} &= \s_1^{-2} \s_2^t && \text{Use $\s_1 \s_2 = \s_2^{-1} \s_1^{-1}$} \\
\s_2^{-1} \s_1^{-2} &= \s_1^{-2} \s_2^t && \text{Invert both sides} \\
\s_1^2 \s_2 &= \s_2^{-t} \s_1^2.
\end{align*}
Now, suppose that $\calP$ is the atomic chiral polyhedron of type $\{2m, m^{\alpha}\}$ whose group is the quotient of $[2m, m^{\alpha}]^+$
by the relations $\s_2^{-1} \s_1 = \s_1^3 \s_2^{1+km^{\alpha-1}}$ and $\s_2 \s_1^{-1} = \s_1^{-3} \s_2^{-1+km^{\alpha-1}}$.
(See \cite[Theorem 4.11]{tight-chiral-polyhedra}.) Then the above discussion shows that this group is equivalent to the quotient of $[2m, m^{\alpha}]^+$
by the relations $\s_2 \s_1^2 = \s_1^2 \s_2^{1+km^{\alpha-1}}$ and $\s_1^2 \s_2 = \s_2^{1-km^{\alpha-1}} \s_1^2$. Furthermore,
the second of those relations is superfluous, since if $\s_2 \s_1^2 = \s_1^2 \s_2^{1+km^{\alpha-1}}$ then
\[ \s_2^{1-km^{\alpha-1}} \s_1^2 = \s_1^2 \s_2^{(1-km^{\alpha-1})(1+km^{\alpha-1})} = \s_1^2 \s_2. \]
So $\G(\calP)$ may be written in the form as it appears in Table~\ref{tab:atomic-polyhedra}.

Next, the proof of \cite[Theorem 3.6]{tight-chiral-polyhedra} shows that $\langle \s_2^{m^{\alpha-1}} \rangle$ is normal and that
the quotient of $\calP$ by this normal subgroup is regular. Thus $X(\calP)$ is a nontrivial subgroup of $\langle \s_2^{m^{\alpha-1}} \rangle$,
and since the latter has prime order $m$, this implies that $X(\calP) = \langle \s_2^{m^{\alpha-1}} \rangle$.

To find a presentation for $\G(\enant{\calP})$, we may change the defining relations of $\G(\calP)$ by replacing $\s_1$ with $\s_1^{-1}$ and
replacing $\s_2$ with $\s_2^{-1}$. This yields:
\begin{align*}
\s_2^{-1} \s_1^{-2} &= \s_1^{-2} \s_2^{-1-km^{\alpha-1}} && \text{Invert both sides} \\
\s_1^2 \s_2 &= \s_2^{1+km^{\alpha-1}} \s_1^2.
\end{align*}
From this we obtain $\s_1^2 \s_2^{1-km^{\alpha-1}} = \s_2^{(1+km^{\alpha-1})(1-km^{\alpha-1})} \s_1^2 = \s_2 \s_1^2$. Thus,
the enantiomorph replaces the parameter $k$ with $-k$ (or equivalently, $m-k$). 

A presentation for the dual of $\calP$ (with respect to the same base flag as $\calP$) 
is obtained by changing each defining relation, replacing $\s_1$ with $\s_2^{-1}$
and $\s_2$ with $\s_1^{-1}$. Applying this to the relation $\s_2 \s_1^2 = \s_1^2 \s_2^{1+km^{\alpha-1}}$ and then inverting both sides
yields $\s_2^2 \s_1 = \s_1^{1+km^{\alpha-1}} \s_2^2$, matching the second row of Table~\ref{tab:atomic-polyhedra}.

This finishes the proof for atomic chiral polyhedra of type $\{2m, m^{\alpha}\}$ and their duals. The proof for the remaining
polyhedra is analogous (referencing \cite[Theorems 3.7 and 3.8]{tight-chiral-polyhedra}), except that for type $\{2^{\beta-1}, 2^{\beta}\}$ and its dual, it is not possible to simplify the presentation
in the same way that we can for the other two cases.
\end{proof}

\begin{corollary}\label{c:qIsPrimePower}
Let $\calP$ be an atomic chiral polyhedron with type $\{p,q\}$ and $p\ge q$. Then $p$ is a prime power.
\end{corollary}

It turns out that the atomic chiral polyhedra satisfy a stronger condition than their definition implies.

\begin{corollary} \label{c:atomic-no-quo}
If $\calP$ is an atomic chiral polyhedron, then it does not properly cover any chiral polyhedron.
\end{corollary}

\begin{proof}
Suppose that $\calP$ is an atomic chiral polyhedron of type $\{p, q\}$, and without loss of generality, assume that $p \geq q$ so that $p$ is a prime power $m^{\alpha}$ (where we could have $m = 2$). 
By the definition of atomic, $\calP$ does not properly cover any chiral polyhedra of type $\{p, q'\}$ or $\{p', q\}$. Furthermore, if $\calQ$ is an orientable rotary polyhedron of type $\{p', q'\}$
where $p'$ is a proper divisor of $p$, then the kernel of the natural map from $\G(\calP)$ to $\G(\calQ)$ contains $\langle \s_1^{m^{\alpha-1}} \rangle$ (see Table~\ref{tab:atomic-polyhedra}),
and since that is the chirality group of $\calP$, it follows that $\calQ$ is regular.
\end{proof}

The following result is an immediate consequence of the definition of atomicity and \cite[Corollary 4.3]{tight-chiral-polyhedra}, which states that every tight chiral polyhedron
of type $\{p, q\}$ covers a tight orientable rotary polyhedron of type $\{p', q\}$ or $\{p, q'\}$.

\begin{proposition} \label{p:nonatomic-quo}
If $\calP$ is a tight chiral polyhedron of type $\{p, q\}$ that is not atomic, then it covers a tight chiral polyhedron of type $\{p', q\}$ or $\{p, q'\}$.
\end{proposition}

When mixing tight orientable rotary polyhedra we may not get a tight structure, as shown next.

\begin{proposition}\label{prop:mixOfAtomic}
Let $\calP$ and $\calQ$ be distinct atomic chiral polyhedra of types $\{p, q\}$ and $\{p, q'\}$, respectively, with $q'$ a divisor of $q$ (not necessarily proper). Then $\calP \lozenge \calQ$ is not tight, regardless of the choice of base flags.
\end{proposition}

\begin{proof}
The mix of $\calP$ and $\calQ$ with respect to any choice of base flags must have type $\{p,q\}$, and if it were tight then it should be isomorphic to $\calP$ and have $\calQ$ as a proper quotient. This is not possible since $\calP$ is atomic.
\end{proof}

We conclude this section with some technical lemmas that allow us to find polytopal quotients of tight orientable rotary $4$-polytopes.

\begin{lemma}\label{l:CoreIsNots2}
Let $\calP$ be a tight chiral polyhedron with type $\{p,q\}$ and $\Gamma(\calP) = \langle \s_1,\s_2 \rangle$. Then $\langle \s_i^2 \rangle$ is not normal in $\Gamma(\calP)$.
\end{lemma}

\begin{proof}
Let $\calQ$ be an atomic chiral polyhedron covered by $\calP$ with automorphism group $\langle \tau_1, \tau_2 \rangle$. If we assume that $\langle \s_i^2 \rangle \triangleleft \Gamma(\calP)$ then by the correspondence theorem in group theory we must also have that $\langle \tau_i^2 \rangle \triangleleft \Gamma(\calQ)$.

It was proven in \cite[Proposition 4.1]{tight-chiral-polyhedra} that if $\calQ$ has type $\{p',q'\}$ and $p'>q'$ then $\langle \tau_1 \rangle$ has a proper subgroup normal in $\Gamma(\calQ)$. On the other hand, it is shown in \cite[Proposition 4.5]{tight-chiral-polyhedra} that either $\langle \tau_1 \rangle$ or $\langle \tau_2 \rangle$ is core-free in $\Gamma(\calQ)$. Up to duality we may assume that $p'>q'$ and hence we only need to show that $\langle \tau_1^2 \rangle$ is not normal in $\Gamma(\calQ)$.

Now, using the classification of atomic chiral polyhedra we see that if $\{p,q\} = \{m^{\alpha}, 2m\}$ then $\langle \tau_1^2 \rangle = \langle \tau_1 \rangle$ and this is not normal in $\Gamma(\calQ)$ (see Lemma \ref{l:NormalThen2}). On the other hand, if $p$ and $q$ are powers of $2$, the dual version of \cite[Lemma 4.13]{tight-chiral-polyhedra} tells us that the core of $\langle \tau_1 \rangle$ is $\langle \tau_1^4 \rangle$. Hence, $\langle \tau_1^2 \rangle$ is not normal in $\Gamma(\calQ)$.
\end{proof}

\begin{lemma}\label{l:s1NotNormal}
Let $\calP$ be a chiral $4$-polytope with chiral facets and let $K$ be the kernel of the action of $\Gamma(\calP)$ on the vertex set. Then $\sigma_i \notin K$ for $i \in \{1,2,3\}$.
\end{lemma}

\begin{proof}
The group $K$ is a normal subgroup of $\G(\calP)$ that is contained in $\langle \s_2, \s_3 \rangle$ since it fixes the
base vertex. The intersection condition (\ref{eq:IntCond2}) implies that $\s_1 \not \in K$.

If $\s_2 \in K$, then also $\s_1 \s_2 \s_1^{-1} \in K$, which implies that $\s_2^{-1} \s_1^{-2} \in K$ and so $\s_1^2 \in K$.
Since $\langle \s_1 \rangle$ has trivial intersection with $K$ (again by the intersection condition), this implies that
$\s_1^2 = \id$, which contradicts Lemma~\ref{l:NoChiral2}.

Similarly, if $\s_3 \in K$, then also $\s_2 \s_3 \s_2^{-1} \in K$, which implies that $\s_3^{-1} \s_2^{-2} \in K$ and so
$\s_2^2 \in K$. It follows that $\s_1^{-1} \s_2^2 \s_1 \in K$. Then $\s_1^{-1} \s_2^2 \s_1$ lies in the intersection
of $\langle \s_1, \s_2 \rangle$ with $\langle \s_2, \s_3 \rangle$, and so it must lie in $\langle \s_2 \rangle$. This
implies that $\langle \s_2^2 \rangle$ is normal in $\langle \s_1, \s_2 \rangle$, contradicting Lemma~\ref{l:CoreIsNots2}.
\end{proof}

	\begin{lemma} \label{lem:vertex-kernel}
	Let $\calP$ be a tight orientable rotary $4$-polytope, with $\G^+(\calP) =
	\langle \s_1, \s_2, \s_{3} \rangle$. Let $K$ be the kernel of the action of $\G^+(\calP)$ on the vertex set. Then:
	\begin{enumerate}
	\item There are integers $a$ and $b$ such that $K = \langle \s_2^{a} \rangle \langle \s_3^{b} \rangle$.
	\item $\calP / K$ is a tight orientable rotary $4$-polytope.
	\end{enumerate}
	\end{lemma}

	\begin{proof}
	Let $a$ be the smallest positive integer such that $\s_2^a \in K$, and let $b$ be the smallest positive integer such
	that $\s_3^b \in K$. (We allow the possibility that $\s_2^a = \id$ or $\s_3^b = \id$.) Let $N = \langle \s_2^a \rangle
	\langle \s_3^b \rangle$. Then clearly $N$ is contained in $K$. To prove the first part,
	it remains to show that $K$ is contained in $N$.
	
	Let $H = \langle \s_2, \s_{3} \rangle$, and suppose that the order of $\s_1$ is $p$. Since $\calP$ is tight,
	it has $p$ vertices, which we can identify with the cosets $H, H \s_1, \ldots, H \s_1^{p-1}$. The action of each
	automorphism on the vertices is by multiplication on the right. 	
	Now, suppose that $\varphi \in K$, which in particular implies that $\varphi \in \langle \s_2, \s_{3} \rangle$.
	Since $\calP$ is tight, Proposition~\ref{prop:tight-group-factoring} implies that
	we may write $\varphi = \s_2^{c} \s_{3}^{d}$. Since $\s_2^c \s_3^d$ fixes all vertices, it follows that
	the action of $\s_2^{c}$ on vertices is the same as the action of $\s_3^{-d}$ on vertices.
	Note that $\s_3^{-1}$ fixes the neighbor of the base vertex in the base edge; namely,
	\[ H \s_1^{-1} \s_3^{-1} = H (\s_3 \s_1)^{-1} = H (\s_1 \s_2^2 \s_3)^{-1} = H \s_1^{-1}. \]
	It follows that $\s_3^{-d}$ fixes that vertex, and thus so does $\s_2^{c}$. However,
	by Corollary \ref{c:FixOneFixAll}, if a power of $\s_2$ fixes a neighbor of the base vertex,
	then it fixes all vertices. Therefore, $\s_2^{c} \in K$, from which it follows that
	$\s_3^{d} \in K$. Then by our choice of $a$ and $b$, it follows that $\varphi \in N$.
	
	The second part follows from the first along with Lemma~\ref{l:s1NotNormal}.
	\end{proof}

\section{Atomic chiral $4$-polytopes with chiral facets and vertex-figures}\label{s:ChiralChiral}

	To understand the structure of tight chiral $4$-polytopes, we use a strategy similar to what was done with
	tight chiral polyhedra. First, we say that a tight chiral $4$-polytope is \emph{atomic} if it does not
	properly cover any tight chiral polytopes. It is clear that every tight chiral polytope covers an atomic
	chiral polytope. Our goal will be to classify the atomic chiral $4$-polytopes.

	By Proposition \ref{prop:facvfChiral} (a), the facets or the vertex-figures of an atomic chiral $4$-polytope must be chiral.
	In this section we classify all atomic chiral $4$-polytopes that have chiral facets and chiral vertex-figures, leaving the case when one of them is regular for Section \ref{s:ChiralRegular}.
	We will show in Theorem \ref{t:ch-ch-atomics} that an atomic chiral $4$-polytope with chiral facets and chiral vertex-figures must have atomic chiral facets and atomic chiral vertex-figures.
	The classification of atomic chiral polyhedra will be then used to find all atomic chiral $4$-polytopes with chiral facets
	and vertex-figures.

\subsection{The structure of atomic chiral $4$-polytopes with chiral facets and vertex-figures}

Now we study atomic chiral $4$-polytopes with chiral facets and chiral vertex-figures. We find several restrictions on atomic chiral $4$-polytopes, culminating in Theorem~\ref{t:ch-ch-atomics}.


	\begin{proposition}\label{prop:ch-ch-props}
	Let $\calP$ be an atomic chiral $4$-polytope of type $\{p, q, r\}$, with chiral facets and vertex-figures, and with $\Gamma(\calP) = \langle \sigma_1, \sigma_2, \sigma_3 \rangle$. Then:
	\begin{enumerate}
	\item $\langle \s_1 \rangle$ and $\langle \s_3 \rangle$ are core-free in $\G(\calP)$.
	\item $q > p$ and $q > r$.
	\end{enumerate}
	\end{proposition}
	
	\begin{proof}
	By duality, for the first part it suffices to prove that $\langle \s_1 \rangle$ is core-free.
	Suppose that $\calP$ is a tight chiral $4$-polytope with chiral facets and vertex-figures,
	and suppose that $\langle \s_1 \rangle$ is not core-free. In other words, there is a normal subgroup
	$N = \langle \s_1^a \rangle$ of $\G(\calP)$. If $\s_1 \in N$, then $\langle \s_1 \rangle$ is normal
	in $\langle \s_1, \s_2 \rangle$, and by Lemma~\ref{l:NormalThen2}, this implies that the facets are regular,
	contradicting our assumptions. So $\s_1 \not \in N$. Then the dual of Proposition~\ref{prop:vfig-quo} shows that 
	$\G(\calP) / N$ is the rotation group of a tight rotary polytope $\calQ$. Since $\langle \s_2, \s_3 \rangle$ has 
	trivial intersection with $N$, the vertex-figures of $\calQ$ must be isomorphic to the vertex-figures of $\calP$, which are chiral.
	Thus $\calQ$ is chiral, which means that $\calP$ is not atomic. This proves part (a).
	
By Proposition \ref{prop:NonTrivialCore}, if $p \ge q$ then there exists a proper divisor $k$ of $p$ such that $\langle \sigma_1^k \rangle \triangleleft \Gamma(\calP)$ contradicting part (a). A dual argument follows if $r \ge q$.
	\end{proof}
	
	\begin{proposition}\label{prop:chirGroups}
	Let $\calP$ be an atomic chiral $4$-polytope of type $\{p, q, r\}$, with chiral facets and vertex-figures, and with $\Gamma(\calP) = \langle \sigma_1, \sigma_2, \sigma_3 \rangle$. Then:
	\begin{enumerate}
	 \item The chirality group $X(\calP)$ is $\langle \sigma_2^{q'} \rangle$ for some $q'$ with $q/q'$ prime,
	 \item The chirality groups of the base facet and vertex-figure are isomorphic to $X(\calP)$.
	\end{enumerate} 
	\end{proposition}

	\begin{proof}
	Let $H$ and $K$ be the kernels of the actions of $\Gamma(\calP)$ on the vertices and on the facets of $\calP$, respectively. By Proposition~\ref{prop:ch-ch-props}(b) together with Corollary \ref{cor:NTkernel} and its dual form, $H \le \langle \sigma_2, \sigma_3 \rangle$ and $K \le \langle \sigma_1, \sigma_2 \rangle$ are non-trivial normal subgroups of $\Gamma(\calP)$. Therefore $H \cap K$ is a normal subgroup of $\Gamma(\calP)$ that by the intersection condition is contained in $\langle \sigma_2 \rangle$.

	Now, Lemma~\ref{lem:vertex-kernel} and its dual show that $\calP/H$ and $\calP/K$ are polytopes, and since $H$ and $K$ are nontrivial and 
	$\calP$ is atomic, $\calP/H$ and $\calP/K$ are regular. Moreover, by Lemma \ref{l:mixquotient}, $\calP/(H\cap K) \cong \calP/H \lozenge \calP/K$ is also regular, implying that $H \cap K$ is non-trivial.

	Since $\calP/(H\cap K)$ is regular, $X(\calP) \le H\cap K = \langle \sigma_2^m \rangle$ for some $m$. If $q/m$ is not prime, then $\langle \sigma_2^{mk} \rangle \triangleleft \Gamma(\calP)$ for any $k$, in particular, for some $k$ such that $q/mk$ is prime. By atomicity of $\calP$, its quotient by $\langle \sigma_2^{mk} \rangle$ is regular and, since it is a maximal quotient, $X(\calP) = \langle \sigma_2^{mk} \rangle$. This concludes part (a).

	Part (b) follows from Part (a) and Lemma \ref{l:chirGroupFacet}.
	\end{proof}

	\begin{proposition}\label{prop:primeFacs}
	Let $\calP$ be an atomic chiral $4$-polytope of type $\{p,q,r\}$ with chiral facets and vertex-figures. 
	If $q$ is a prime power then the facets and vertex-figures of $\calP$ are atomic chiral polyhedra.
	\end{proposition}

	\begin{proof}
	Suppose $\calP$ has facets isomorphic to $\calQ_1$ and vertex-figures isomorphic to $\calQ_2$.
	By Proposition~\ref{prop:chirGroups}, $X(\calP) = X(\calQ_1)=X(\calQ_2)$, and these groups are cyclic of prime order.
	If $q$ is a prime power then $X(\calP)$ is contained in all proper subgroups of $\langle \sigma_2 \rangle$, and so
	$\calQ_1$ does not cover any tight chiral polyhedra of type $\{p, q'\}$ with $q'$ a proper divisor of $q$.
	Proposition~\ref{prop:ch-ch-props} says that $\langle \s_1 \rangle$ is core-free, and so also $\calQ_1$ does not
	cover any tight chiral polyhedra of type $\{p', q\}$ with $p'$ a proper divisor of $p$. It follows that $\calQ_1$
	is atomic, and a dual argument proves that $\calQ_2$ is atomic as well.
	\end{proof}
	
	We are now ready to prove the main necessary condition for a tight chiral $4$-polytope with chiral facets and chiral vertex-figures to be atomic.

	\begin{theorem} \label{t:ch-ch-atomics}
	If $\calP$ is an atomic chiral $4$-polytope with chiral facets and chiral vertex-figures then the facets and vertex-figures are atomic chiral polyhedra.
	\end{theorem}
		
	\begin{proof}
	Assume that $\calP$ has type $\{p,q,r\}$. The facets and vertex-figures of $\calP$ are isomorphic to some chiral polyhedra 
	$\calQ_1$ and $\calQ_2$, respectively. Proposition \ref{prop:chirGroups} (b) tells us that $X(\calP) = \langle \sigma_2^{q'} \rangle$ with $q/q'$ prime and that $X(\calP) = X(\calQ_1) = X(\calQ_2)$.

	Assume to the contrary that $\calQ_i$ is not atomic for some $i \in \{1,2\}$. Then, by Proposition \ref{prop:primeFacs}, $q$ must have at least two distinct prime factors, which by Corollary~\ref{c:qIsPrimePower} implies that neither $\calQ_1$ nor $\calQ_2$ is atomic. Let $m = q/q'$, which is prime (but not necessarily odd). Then $q = m^{\alpha} t$ for some $\alpha$ and some $t$ not divisible by $m$.

	Since $\calQ_1$ is not atomic there exists $N_1 \triangleleft \langle \s_1,\s_2 \rangle$ such that $\calQ_1/N_1$ is an atomic chiral polyhedron. By Lemma \ref{l:subgroups} there exist $a$ and $b$ such that $N_1 = \langle \s_1^{a} \rangle \langle \s_2^{b} \rangle$ and $\calQ_1 / N_1$ has type $\{a,b\}$.

	Now $X(\calQ_1)=X(\calP)$, and therefore $\sigma_2^{q'} \notin N_1$. It follows that $q/b$ divides $t$ and $m^{\alpha}$ divides $b$. Lemma \ref{l:ChirGroupS2} implies that $X(\calQ_1/N_1)$ is contained in the subgroup generated by the second standard generator of $\Gamma(\calQ_1/N_1)$. Since $\calQ_1/N_1$ is atomic, we can conclude that $a < b$ by the classification of atomic chiral polyhedra. Corollary \ref{c:qIsPrimePower} now tells us that $b = m^{\alpha}$.

	We proceed in a dual manner to observe that there exists $N_2= \langle \s_2^{b'}\rangle \langle \s_3^{c} \rangle \le \langle \s_2,\s_3 \rangle$ such that $\calQ_2/N_2$ is an atomic chiral polyhedron with type $\{b',c\} = \{m^{\alpha},c\}$. In particular, $b=b'$.

	Let $K= \langle \s_1^{a} \rangle \langle \s_2^{b} \rangle \langle \s_3^{c} \rangle$.
	We claim that $K\triangleleft \Gamma(\calP)$. 
	To see this, note that
	\[\sigma_2^{-1}(\sigma_1^{k_1 a} \sigma_2^{k_2 b} \sigma_3^{k_3 c})\sigma_2 = (\sigma_2^{-1} \sigma_1^{k_1 a} \sigma_2^{k_2 b} \sigma_2)(\sigma_2^{-1} \sigma_3^{k_3 c} \sigma_2) \in (\langle \sigma_1^{a} \rangle \langle \sigma_2^{b} \rangle)(\langle \sigma_2^{b} \rangle \langle \sigma_3^{c} \rangle),\]
	and as noted in the proof of Lemma \ref{l:rels1s3},
	\[
	\sigma_3(\sigma_1^{k_1 a} \sigma_2^{k_2 b} \sigma_3^{k_3 c})\sigma_3^{-1} = (\sigma_2^{-1}(\sigma_1^{-k_1 a})\sigma_2)(\sigma_3 \sigma_2^{k_2 b}\sigma_3^{k_3 c} \sigma_3^{-1})
	\in (\langle \sigma_1^{a} \rangle \langle \sigma_2^{b} \rangle)(\langle \sigma_2^{b} \rangle \langle \sigma_3^{c} \rangle).\]
	A dual argument shows that $K$ is invariant under conjugation by $\sigma_1$.
	
	Now, Lemma \ref{l:s1NotNormal} and Proposition~\ref{prop:vfig-quo} imply that $\calP/K$ is a polytope, and since $\calP$ is atomic this polytope must be regular of type $\{a, b, c\}$. In particular, this implies that the facets are regular polyhedra of type $\{a, b\}$. On the other hand, the facets must be a quotient of $\calQ_1/N_1$, which is a tight chiral polyhedron of type $\{a, b\}$. But no tight polyhedron properly covers another polyhedron of the same type, and so we have a contradiction.
	\end{proof}

	Now let us show that the conditions in Proposition~\ref{prop:ch-ch-props} and Theorem~\ref{t:ch-ch-atomics} suffice
	if we want to build an atomic chiral $4$-polytope.
	
	\begin{corollary}
	\label{cor:atomic-ch-ch}
	A tight chiral $4$-polytope $\calP$ with chiral facets and vertex-figures is atomic if and only if
	\begin{enumerate}
	\item The facets and vertex-figures are atomic, and
	\item $\langle \s_1 \rangle$ and $\langle \s_3 \rangle$ are core-free in $\G(\calP)$.
	\end{enumerate}
	\end{corollary}
	
	\begin{proof}
	Theorem~\ref{t:ch-ch-atomics} and Proposition~\ref{prop:ch-ch-props} prove that the conditions are necessary. 
	Now, suppose that $\calP$ satisfies the conditions. If $\calQ$ is a proper chiral quotient of $\calP$, then
	$\calQ$ is still tight, and so Proposition~\ref{prop:facvfChiral} says that either the facets or vertex-figures are chiral.
	Without loss of generality, suppose that the facets of $\calQ$ are chiral. The facets of $\calP$ cover
	the facets of $\calQ$, and since the facets of $\calP$ are atomic, this implies that $\calQ$ has the same facets.
	In particular, if $\calP$ has type $\{p, q, r\}$, then $\calQ$ has type $\{p, q, r'\}$ for some $r'$ dividing $r$.
	By tightness, $|\G(\calP)| = pqr$ and $|\G(\calQ)| = pqr'$, and so since $\calQ$ is a proper quotient of $\calP$,
	we have $r' \neq r$. Furthermore, $\G(\calQ) = \G(\calP) / \langle \s_3^{r'} \rangle$. But this contradicts
	that $\langle \s_3 \rangle$ is core-free in $\G(\calP)$.
	\end{proof}
	
\subsection{Classification of atomic chiral $4$-polytopes with chiral facets and vertex-figures}

	In light of Lemma \ref{l:flatflat}, once we know the possible types of facets and vertex-figures of an atomic
	chiral $4$-polytope, all we need to do is try amalgamating the compatible pairs and see which ones give us
	a group of the proper size that satisfies the intersection condition. Theorem~\ref{t:ch-ch-atomics} implies
	that the facets and vertex-figures must appear on Table~\ref{tab:atomic-polyhedra}. Combined with Proposition
	\ref{prop:ch-ch-props}, we find that the automorphism group of an atomic chiral $4$-polytope with chiral facets and vertex-figures
	must be one of the groups in Table~\ref{tab:atomic-cc}. For simplicity we avoid including the various parameters (such as $m$, $\alpha$, and $k_1$) in the names of the groups. The ``extra relations'' show how to define the group as a quotient of the given parent group.
	
	\begin{table}[htbp]
	\centering
	\begin{tabular}{|l|l|l|l|} \hline
	Group name & Parent group & Extra relations & Notes \\ \hline
	$\G_1$ & $[2m, m^{\alpha}, 2m]^+$ & $\s_2 \s_1^2 = \s_1^2 \s_2^{1+k_1 m^{\alpha-1}}$ & $m$ odd prime, $\alpha \geq 2$, \\	
	 & & $\s_3^2 \s_2 = \s_2^{1+k_2 m^{\alpha-1}} \s_3^2$ & $k_1, k_2 \in \{1, \ldots, m-1\}$ \\ \hline

	$\G_2$ & $[8, 2^{\beta}, 8]^+$ & $\s_2 \s_1^2 = \s_1^2 \s_2^{1+\epsilon_1 2^{\beta-2}}$ & $\beta \geq 5$, $\epsilon_1, \epsilon_2 \in \{-1, 1\}$ \\
	 & & $\s_3^2 \s_2 = \s_2^{1+\epsilon_2 2^{\beta-2}} \s_3^2$ & \\ \hline

	$\G_3$ & $[2^{\beta-1}, 2^{\beta}, 2^{\beta-1}]^+$ & $\s_2^{-1} \s_1 = \s_1^{-1+2^{\beta-2}} \s_2^{-3+\epsilon_1 2^{\beta-2}}$ & $\beta \geq 5$, $\epsilon_1, \epsilon_2 \in \{-1, 1\}$ \\
	& & $\s_2 \s_1^{-1} = \s_1^{1+2^{\beta-2}} \s_2^{3+\epsilon_1 2^{\beta-2}}$ & \\
	& & $\s_3^{-1} \s_2 = \s_2^{3 + \epsilon_2 2^{\beta-2}} \s_3^{1+2^{\beta-2}}$ & \\
	& & $\s_3 \s_2^{-1} = \s_2^{-3 + \epsilon_2 2^{\beta-2}} \s_3^{-1+2^{\beta-2}}$ & \\ \hline

	$\G_4$ & $[8, 2^{\beta}, 2^{\beta-1}]^+$ & $\s_2 \s_1^2 = \s_1^2 \s_2^{1+\epsilon_1 2^{\beta-2}}$ & $\beta \geq 5$, $\epsilon_1, \epsilon_2 \in \{-1, 1\}$ \\
	& & $\s_3^{-1} \s_2 = \s_2^{3 + \epsilon_2 2^{\beta-2}} \s_3^{1+2^{\beta-2}}$ & \\
	& & $\s_3 \s_2^{-1} = \s_2^{-3 + \epsilon_2 2^{\beta-2}} \s_3^{-1+2^{\beta-2}}$ & \\ \hline
	\end{tabular}
	\caption{The possible groups of atomic chiral $4$-polytopes with chiral facets and vertex-figures}
	\label{tab:atomic-cc}
	\end{table}
		
	Using GAP \cite{GAP}, we verified that $\G_2, \G_3,$ and $\G_4$ have the correct order and satisfy the intersection
	condition for $\beta = 5$ and $\beta = 6$, and for all four choices of $(\epsilon_1, \epsilon_2)$.
	Thus, for these parameter values, the group is the automorphism group of a tight chiral
	polytope. We similarly verified that $\G_1$ is the automorphism group of a tight chiral polytope
	for $m = 3, \alpha \in \{2,3\}, k_1 = k_2 \in \{1,2\}$ and for $m = 5, \alpha = 2, k_1 = k_2 \in \{1, \ldots, 4\}$.
	Furthermore, for these values of $m$ and $\alpha$, we verified that $\G_1$ does not have the proper order when $k_1 \neq k_2$,
	and so does not define the automorphism group of a tight chiral polytope.
		
	For the group $\G_1$, we will show that we do in fact need $k_1 = k_2$. Then, for each group
	we will describe a permutation representation of the group. There is a standard strategy that we used to
	determine the permutation representation, based on the following facts.
	If $\calP$ is a tight chiral $4$-polytope of type $\{p, q, r\}$, then the cosets of $\langle \s_1 \rangle$ are of the form
	$\langle \s_1 \rangle \s_2^b \s_3^c$, and $\G(\calP)$ acts on the set of cosets by right multiplication.
	Furthermore, since $\G(\calP)$ is tight, then for every $i$ we can rewrite $\langle \s_1 \rangle \s_2^b \s_3^c \s_i$ as 
	$\langle \s_1 \rangle \s_2^{b'} \s_3^{c'}$ for some $b'$ and $c'$. So for each $i$, we determined how $b'$ and $c'$
	depend on $b$ and $c$. We then encode the coset $\langle \s_1 \rangle \s_2^b \s_3^c$ as the pair $(b,c) \in \Z_q \times \Z_r$
	and write down a description of the multiplication.
	
	Once we have a permutation representation, the following lemma will show that we indeed have found the group
	of a tight chiral polytope.

	\begin{lemma}
	\label{lem:perm-rep}
	Suppose that $\calP$ is a tight orientable rotary polyhedron and that $\calQ$ is a tight chiral polyhedron.
	Let $\G = \langle \s_1, \s_2, \s_3 \rangle = [p,q,r]^+/N_3$, the amalgamation of $\G^+(\calP)$ with $\G^+(\calQ)$
	 as defined in Lemma~\ref{l:flatflat}. 
	Suppose that there is a permutation group $G = \langle \pi_1, \pi_2, \pi_3 \rangle$ on $\Z_q \times \Z_r$ such that
	the function that sends each $\s_i$ to $\pi_i$ determines a group epimorphism. Further, suppose that:
	\begin{enumerate}
	\item $\pi_1$ fixes $(0,0)$.
	\item There is some point $(b,c)$ such that the smallest power of $\pi_1$ that fixes $(b,c)$ is $\pi_1^p$.
	\item $(b, 0) \pi_2 = (b+1, 0)$ for all $b$.
	\item $(b,c) \pi_3 = (b, c+1)$ for all $b$ and $c$.
	\end{enumerate}
	Then $\G \cong G$, and $\G$ is the rotation group of a tight chiral polytope of type $\{p, q, r\}$.
	\end{lemma}
	
	\begin{proof}
	First, note that since $G$ is a quotient of $\G$, then $\pi_1^p = \pi_2^q = \pi_3^r = \id$.
	The given conditions then imply that no smaller powers of any $\pi_i$ will equal the identity. Now, since $\G$ is a tight quotient of $[p,q,r]^+$, it follows that
	$|G| \leq |\G| \leq pqr$. If we can show that $G$ satisfies the intersection condition, then Proposition~\ref{prop:tight-group-factoring} will imply that $|G| = pqr$, and thus that $G \cong \G$ and that $\G$ is the rotation group of a tight orientable rotary polytope of type $\{p, q, r\}$. Furthermore, by Lemma~\ref{l:flatflat}, such a polytope will have chiral vertex-figures isomorphic to $\calQ$ and thus
	it will be chiral itself.
	
	To show that $G$ satisfies the intersection condition, we first need to show that
	\[ \langle \pi_1 \rangle \cap \langle \pi_2 \rangle = \{ \id \} = \langle \pi_2 \rangle \cap \langle \pi_3 \rangle. \]
	If $\varphi = \pi_1^a = \pi_2^b$, then
	\[ (0, 0) = (0, 0) \pi_1^a = (0, 0) \pi_2^b = (b, 0), \]
	and so $b \equiv 0$ (mod $q$), which implies that $\varphi$ is trivial.
	Similarly, if $\varphi = \pi_2^b = \pi_3^c$, then
	\[ (b, 0) = (0, 0) \pi_2^b = (0, 0) \pi_3^c = (0, c), \]
	which implies that $\varphi$ is trivial. Finally, we need to show that
	\[ \langle \pi_1, \pi_2 \rangle \cap \langle \pi_2, \pi_3 \rangle. \]
	Consider $\varphi$ in this intersection. Since $G$ is a quotient of the tight group $\G$, we may write
	$\varphi = \pi_1^a \pi_2^b = \pi_2^{b'} \pi_3^c$ for some $a, b, b', c$. We have
	\[ (0, 0) \pi_1^a \pi_2^b = (0, 0) \pi_2^b = (b, 0) \]
	and
	\[ (0, 0) \pi_2^{b'} \pi_3^c = (b', 0) \pi_3^c = (b', c). \]
	It follows that $c \equiv 0$ (mod $r$) and thus that $\pi_3^c = \id$.
	So $\varphi = \pi_2^{b'} \in \langle \pi_2 \rangle$, as desired.
	\end{proof}

	\begin{theorem} \label{t:2m-ma-2m}
	The group $\G_1$ is the automorphism group of an atomic chiral $4$-polytope of type $\{2m, m^{\alpha}, 2m\}$ 
	if and only if $k_1 = k_2$.
	\end{theorem}
	
	\begin{proof}
	First let us show that $k_1 = k_2$.
	Note that
	\begin{align*}
	\s_1 \s_2^{k_1 m^{\alpha-1}} &= \ul{\s_1 \s_2^{-1}} \s_2^{1+k_1 m^{\alpha-1}} \\
	&= \s_2^{1-k_1 m^{\alpha-1}} \ul{\s_1^3 \s_2^{1+k_1 m^{\alpha-1}}} \\
	&= \s_2^{1-k_1 m^{\alpha-1}} \s_2^{-1} \s_1 \\
	&= \s_2^{-k_1 m^{\alpha-1}} \s_1 \\
	\end{align*}
	Thus, conjugation by $\s_1$ inverts $\s_2^{k_1 m^{\alpha-1}}$, and since $1 \leq k_1 \leq m-1$ and $m$
	is prime, this implies that conjugation by $\s_1$ inverts $\s_2^{m^{\alpha-1}}$. A similar argument
	shows that conjugation by $\s_3$ inverts $\s_2^{m^{\alpha-1}}$.
	Then, using Lemma \ref{l:rels1s3}(a) we see that
	\begin{align*}
	\s_3 \underline{\s_2^{-1} \s_1} &= \underline{\s_3 \s_1^3} \s_2^{1+k_1 m^{\alpha-1}} \\
	&= \s_2^{-1} \s_1^{-3} \underline{\s_2 \s_3 \s_2} \s_2^{k_1 m^{\alpha-1}} \\
	&= \s_2^{-1} \s_1^{-3} \underline{\s_3^{-1} \s_2^{k_1 m^{\alpha-1}}} \\
	&= \underline{\s_2^{-1} \s_1^{-3}} \s_2^{-k_1 m^{\alpha-1}} \s_3^{-1} \\
	&= \s_1^{-1} \s_2^{1-2k_1 m^{\alpha-1}} \s_3^{-1}.
	\end{align*}
	On the other hand,
	\begin{align*}
	\underline{\s_3 \s_2^{-1}} \s_1 &= \s_2^{1+k_2 m^{\alpha-1}} \underline{\s_3^3 \s_1} \\
	&= \s_2^{k_2 m^{\alpha-1}} \underline{\s_2 \s_1 \s_2} \s_3^{-3} \s_2^{-1} \\
	&= \underline{\s_2^{k_2 m^{\alpha-1}} \s_1^{-1}} \s_3^{-3} \s_2^{-1} \\
	&= \s_1^{-1} \s_2^{-k_2 m^{\alpha-1}} \underline{\s_3^{-3} \s_2^{-1}} \\
	&= \s_1^{-1} \s_2^{1 - 2k_2 m^{\alpha-1}} \s_3^{-1}. \\
	\end{align*}
	Thus $\s_2^{1-2k_1 m^{\alpha-1}} = \s_2^{1-2k_2 m^{\alpha-1}}$, and 
	since $k_1$ and $k_2$ are defined modulo $m$ (which is an odd prime), it follows that $k_1 = k_2$.

	Now, fix $k_1 = k_2 = k$. Let $D = k m^{\alpha-1}$. For $b \in \Z_{m^{\alpha}}$, we define $\overline{b} = \ds -b + \frac{b(b-1)}{2} D$. 
	Then we define permutations of $\Z_{m^{\alpha}} \times \Z_{2m}$ as follows:
	\begin{align*}
	(b, c) \pi_1 &= \begin{cases}
	(\overline{b} + \frac{c}{2} D, -c), & \textrm{ if $c$ is even}, \\
	(\overline{b} + 2 - \frac{c-1}{2} D, 2-c), & \textrm{ if $c$ is odd,}
	\end{cases} \\
	(b, c) \pi_2 &= \begin{cases}
	(b + 1 + \frac{c}{2} D, c), & \textrm{ if $c$ is even}, \\
	(b - 1 - \frac{c-1}{2} D, c-2), & \textrm{ if $c$ is odd,}
	\end{cases} \\
	(b, c) \pi_3 &= (b, c+1).
	\end{align*}
	We want to show that $\langle \pi_1, \pi_2, \pi_3 \rangle$ satisfies the defining relations of $\G_1$. Here are
	several intermediate calculations; the first three formulas help verify the fourth and fifth.
	\begin{enumerate}
	\item $\ol{b} D = -bD$
	\item $\ol{b + tD} = \ol{b} - tD$
	\item $\ol{\ol{b}} = b(1+D)$
	\item $(b, c) \pi_1 \pi_2 = (\ol{b}+1, -c)$
	\item $(b, c) \pi_1^2 = \begin{cases}
	(b(1+D) - cD, c) & \textrm{ if $c$ is even}, \\
	(b(1-D) + cD, c) & \textrm{ if $c$ is odd}.
	\end{cases}$
	\item $(b, c) \pi_2^m = \begin{cases}
	(b+m, c), & \textrm{ if $c$ is even}, \\
	(b-m, c), & \textrm{ if $c$ is odd}.
	\end{cases}	$
	\end{enumerate}
	From the above, it is straightforward to show that
	\[ (b, c) \pi_1^{2t} = \begin{cases}
	(b(1+tD) - tcD, c) & \textrm{ if $c$ is even}, \\
	(b(1-tD) + tcD, c) & \textrm{ if $c$ is odd}.
	\end{cases} \]
	Then $(b, c) \pi_1^{2m} = (b,c)$ since $mD \equiv 0$ (mod $m^{\alpha}$). We note that the action of $\pi_1$ on the second coordinate makes it clear that $\pi_1$ has even order, and for $1 \leq t \leq m-1$ we have $(1, 0) \pi_1^{2t} = (1+tD, 0) \neq (1, 0)$. So $\pi_1$ has order $2m$ (and not a proper divisor).
	
	From the sixth calculation above, it is clear that $\pi_2^{m^{\alpha}} = \id$. It's also clear that
	$\pi_3$ has order $2m$.

	Next, we want to show that $(\pi_1 \pi_2)^2 = (\pi_1 \pi_2 \pi_3)^2 = \id$. Since $(b, c) \pi_1 \pi_2 = 
	(\ol{b} + 1, -c)$, we have:
	\begin{align*}
	(b, c) (\pi_1 \pi_2)^2 &= (\ol{b} + 1, -c) \pi_1 \pi_2 \\
	&= (\ol{\ol{b} + 1} + 1, c), \\
	\end{align*}
	and
	\begin{align*}
	\ol{\ol{b} + 1} + 1 &= -(\ol{b}+1) + \frac{(\ol{b}+1)\ol{b}}{2}D + 1 \\
	&= -\ol{b} + \frac{(-b+1)(-b)}{2}D \\
	&= b - \frac{b(b-1)}{2}D + \frac{(-b+1)(-b)}{2}D \\
	&= b.
	\end{align*}
	So $(b, c) (\pi_1 \pi_2)^2 = (b, c)$. Essentially the same proof shows that $(b, c) (\pi_1 \pi_2 \pi_3)^2 = (b, c)$.
	Verifying that $(b, c) (\pi_2 \pi_3)^2 = (b,c)$ is straightforward.
	
	Finally, verifying that $\pi_2 \pi_1^2 = \pi_1^2 \pi_2^{1+D}$ and $\pi_3^2 \pi_2 = \pi_2^{1+D} \pi_3^2$ is
	relatively straightforward with the hints above. Lemma~\ref{lem:perm-rep} and Corollary \ref{cor:atomic-ch-ch}
	then finish the proof.
	\end{proof}
	
	\begin{theorem} \label{t:8-2b-8}
	The group $\G_2$ is the automorphism group of an atomic chiral $4$-polytope of type $\{8, 2^{\beta}, 8\}$ 
	for all four choices of $(\epsilon_1, \epsilon_2)$ and for every $\beta \geq 5$.
	\end{theorem}
	
	\begin{proof}
	Let $D = 2^{\beta-3}$.
	For $b \in \Z_{2^{\beta}}$, we define $\overline{b} = -b + b(b-1)D\epsilon_1$. Then we define permutations
	of $\Z_{2^{\beta}} \times \Z_8$ as follows:
	\begin{align*}
	(b, c) \pi_1 &= \begin{cases}
	(\overline{b} + D \epsilon_2 c, -c), & \textrm{ if $c$ is even}, \\
	(\overline{b} + 2 - D \epsilon_2 (c-1), 2-c), & \textrm{ if $c$ is odd,}
	\end{cases} \\
	(b, c) \pi_2 &= \begin{cases}
	(b + 1 + D \epsilon_2 c, c), & \textrm{ if $c$ is even}, \\
	(b - 1 - D \epsilon_2 (c-1), c-2), & \textrm{ if $c$ is odd,}
	\end{cases} \\
	(b, c) \pi_3 &= (b, c+1).
	\end{align*}
	The following intermediate calculations can be used to verify that there is a well-defined epimorphism from $\G_2$ to $\langle \pi_1, \pi_2, \pi_3 \rangle$ sending each $\s_i$ to $\pi_i$.
	\begin{enumerate}
	\item $4D \equiv 2^{\beta-1}$ and $8D \equiv 0$ (mod $2^{\beta}$)
	\item If $\beta = 5$, then $D^2 \equiv 2^{\beta-1}$ (mod $2^{\beta}$),
	and if $\beta \geq 6$ then $D^2 \equiv 0$ (mod $2^{\beta}$).
	\item $\ol{b} D = -bD$
	\item $\ol{b+2tD} = \ol{b} - 2tD$ for all $t$
	\item $\ol{\ol{b}} = b(1+2D \epsilon_1)$
	\item $(b, c) \pi_1 \pi_2 = (\ol{b}+1, -c)$
	\item $(b, c) \pi_1^2 = \begin{cases}
	(b(1+2D \epsilon_1) -2D \epsilon_2 c, c) & \textrm{ if $c$ is even}, \\
	(b(1-2D \epsilon_1) + 2D \epsilon_1 + 2D \epsilon_2(c-1), c) & \textrm{ if $c$ is odd}.
	\end{cases}$
	\item $(b, c) \pi_2^8 = \begin{cases}
	(b+8, c), & \textrm{ if $c$ is even}, \\
	(b-8, c), & \textrm{ if $c$ is odd}.
	\end{cases}$
	\end{enumerate}
	We omit the details of showing that $\langle \pi_1, \pi_2, \pi_3 \rangle$ satisfies the defining relations
	of $\G_2$. Lemma~\ref{lem:perm-rep} and Corollary \ref{cor:atomic-ch-ch}
	then finish the proof.
	\end{proof}

	\begin{theorem} \label{t:2b-2bp1-2b}
	The group $\G_3$ is the automorphism group of an atomic chiral $4$-polytope of type $\{2^{\beta-1}, 2^{\beta}, 2^{\beta-1}\}$ 
	for all four choices of $(\epsilon_1, \epsilon_2)$ and for every $\beta \geq 5$.
	\end{theorem}
	
	\begin{proof}
	Let $D = 2^{\beta-3}$.
	For $b \in \Z_{2^{\beta}}$, we define 
	\[ \overline{b} = \begin{cases}
	b(1+ D \epsilon_1) & \textrm{ if $b$ is even}, \\
	(b-1)(1-D\epsilon_1) - 1 & \textrm{ if $b$ is odd.}
	\end{cases} \]
	Then we define permutations
	of $\Z_{2^{\beta}} \times \Z_{2^{\beta-1}}$ as follows:
	\begin{align*}
	(b, c) \pi_1 &= \begin{cases}
	(\overline{b} + 2c + D \epsilon_2 c, c(D+1)) & \textrm{ if $c$ is even,} \\
	(\overline{b} + 2c - D \epsilon_2 (c-1), c(D+1)-D) & \textrm{ if $c$ is odd,} \\
	\end{cases} \\
	(b, c) \pi_2 &= \begin{cases}
	(b + 1 - 2c + D \epsilon_2 c, c(D-1)) & \textrm{ if $c$ is even,} \\
	(b + 1 - 2c - D \epsilon_2 (c-1), c(D-1)-D) & \textrm{ if $c$ is odd,} \\
	\end{cases} \\
  	(b, c) \pi_3 &= (b, c+1).
	\end{align*}

	The following intermediate calculations can be used to verify that there is a well-defined epimorphism from $\G_3$ to $\langle \pi_1, \pi_2, \pi_3 \rangle$ sending each $\s_i$ to $\pi_i$.
	\begin{enumerate}
	\item $4D \equiv 2^{\beta-1}$ and $8D \equiv 0$ (mod $2^{\beta}$)
	\item If $\beta = 5$, then $D^2 \equiv 2^{\beta-1}$ (mod $2^{\beta}$),
	and if $\beta \geq 6$ then $D^2 \equiv 0$ (mod $2^{\beta}$).
	\item $(b, c) \pi_1 \pi_2 = (\ol{b}+1, -c)$
	\item $\ol{\ol{b}+1} = b - 1$
	\item $(b, c) \pi_2^4 = \begin{cases}
	(b+4, c) & \textrm{ if $c$ is even } \\
	(b+4 + 4D, c) & \textrm{ if $c$ is odd}.
	\end{cases}$
	\item $(b, c) \pi_2^8 = (b+8, c)$.
	\item $(b, c) \pi_1^2 = \begin{cases}
	(b(1+2D \epsilon_1) + 2c(2+D \epsilon_1), c) & \textrm{ if $b$ is even}, \\
	(b(1-2D \epsilon_1) + 2c(2-D \epsilon_1) + 4D - 4, c) & \textrm{ if $b$ is odd.} \\	
	\end{cases}$
	\item $(b, c) \pi_1^4 = \begin{cases}
	(b + c(4D+8), c) & \textrm{ if $b$ is even}, \\
	(b + (c-1)(4D+8), c) & \textrm{ if $b$ is odd.}
	\end{cases}$
	\item $(b, c) \pi_1^{2^{\beta-2}} = (b + 4D(b+c), c) = 
	\begin{cases}
	(b, c) & \textrm{ if $b$ and $c$ have the same parity}, \\
	(b + 4D, c) & \textrm{ if $b$ and $c$ have opposite parity.}
	\end{cases}$
	\end{enumerate}

	Here we give more details on how to verify that $\langle \pi_1, \pi_2, \pi_3 \rangle$ satisfies the extra relations from Table~\ref{tab:atomic-cc}.
	To verify the relation $\pi_2^{-1} \pi_1 = \pi_1^{-1+2^{\beta-2}} \pi_2^{-3 + \epsilon_1 2^{\beta-2}}$,
	we rewrite it:
	\begin{align*}
	\pi_2^{-1} \pi_1 &= \pi_1^{-1+2^{\beta-2}} \pi_2^{-3 + \epsilon_1 2^{\beta-2}} & & \text{Multiply by $\pi_1^{-1}$ on the left} \\
	\pi_1^{-1} \pi_2^{-1} \pi_1 &= \pi_1^{-2+2^{\beta-2}} \pi_2^{-3 + \epsilon_1 2^{\beta-2}} & & \text{$\pi_1^{-1} \pi_2^{-1} = \pi_2 \pi_1$} \\
	\pi_2 \pi_1^2 &= \pi_1^{-2+2^{\beta-2}} \pi_2^{-3 + \epsilon_1 2^{\beta-2}} & & \text{Multiply by $\pi_1^2$ on the left and $\pi_2^4$ on the right} \\
	\pi_1^2 \pi_2 \pi_1^2 \pi_2^4&= \pi_1^{2^{\beta-2}} \pi_2^{1 + \epsilon_1 2^{\beta-2}} & & 
	\end{align*}
	Then we can show that both sides send $(b,c)$ to
	\[ \begin{cases}
	(b + 4Db + 2D\epsilon_1 + 1 - 2c + D\epsilon_2 c, c(D-1)) & \textrm{ if $c$ is even,} \\
	(b + 4Db - 2D\epsilon_1 + 1 - 2c - D\epsilon_2 (c-1), c(D-1) - D) & \textrm{ if $c$ is odd.}
	\end{cases} \]	
	After showing that that relation holds, we can use it to rewrite the second relation into a form that is easier to verify:
	\begin{align*}
	\pi_2 \pi_1^{-1} &= \pi_1^{1+2^{\beta-2}} \pi_2^{3+\epsilon_1 2^{\beta-2}} & & \text{Multiply by $\pi_1^{-1}$ on the left} \\
	\pi_1^{-1} \pi_2 \pi_1^{-1} &= \pi_1^{2^{\beta-2}} \pi_2^{3+\epsilon_1 2^{\beta-2}} & & \text{Rewrite using first relation} \\
	\pi_2^{3-\epsilon_1 2^{\beta-2}} \pi_1^{-2^{\beta-2}} &= \pi_1^{2^{\beta-2}} \pi_2^{3+\epsilon_1 2^{\beta-2}} & & \text{Multiply by $\pi_2$ on the left and right} \\
	\pi_2^{4-\epsilon_1 2^{\beta-2}} \pi_1^{-2^{\beta-2}} \pi_2 &= \pi_2 \pi_1^{2^{\beta-2}} \pi_2^{4+\epsilon_1 2^{\beta-2}} 
	\end{align*}
	Then we can show that both sides send $(b,c)$ to
	\[ \begin{cases}
	(b + 4Db - 2D\epsilon_1 + 5 - 2c + D\epsilon_2 c, c(D-1)) & \textrm{ if $c$ is even,} \\
	(b + 4Db - 2D\epsilon_1 + 5 - 2c - D\epsilon_2 (c-1), c(D-1) - D) & \textrm{ if $c$ is odd.}
	\end{cases} \]	
	To verify the third relation, we rewrite it as $\pi_2 \pi_3^{-1} \pi_2 = \pi_2^{4+2 \epsilon_2 D} \pi_3^{1+2D}$.
	Then we can show that both sides send $(b,c)$ to
	\[ \begin{cases}
	(b+4+2 \epsilon_2 D, c+1+2D) & \textrm{ if $c$ is even,} \\
	(b+4-2 \epsilon_2 D, c+1+2D) & \textrm{ if $c$ is odd.}
	\end{cases} \]
	To verify the fourth relation, we multiply both sides by $\pi_2^4$ on the left and $\pi_2$ on the right to obtain
	$\pi_2^4 \pi_3 = \pi_2^{1+2\epsilon_2 D} \pi_3^{-1+2D} \pi_2$.
	Then we can show that both sides send $(b,c)$ to
	\[ \begin{cases}
	(b+4, c+1) & \textrm{ if $c$ is even,} \\
	(b+4+4D, c+1) & \textrm{ if $c$ is odd.}
	\end{cases} \]
	Lemma~\ref{lem:perm-rep} and Corollary \ref{cor:atomic-ch-ch}
	then finish the proof.	
	\end{proof}

	\begin{theorem} \label{t:8-2bp1-2b}
	The group $\G_4$ is the automorphism group of an atomic chiral $4$-polytope of type $\{8, 2^{\beta}, 2^{\beta-1}\}$ 
	for all four choices of $(\epsilon_1, \epsilon_2)$ and for every $\beta \geq 5$.
	\end{theorem}
	
	\begin{proof}
	We use the same permutation representation as
	Theorem~\ref{t:2b-2bp1-2b}, except that we now define $\overline{b} = -b + b(b-1) D \epsilon_1$ as in Theorem~\ref{t:8-2b-8}.
	Note that since the relations of $\G_4$ that involve only $\s_2$ and $\s_3$ are the same as the relations in $\G_3$,
	and the permutation representation for those two elements is the same, the only relations that need to be verified
	are those that include $\s_1$. Here are some intermediate calculations:
	\begin{enumerate}
	\item $(b, c) \pi_1^2 = \begin{cases}
	(b(1+2D \epsilon_1), c) & \textrm{ if $c$ is even}, \\
	(b(1-2D \epsilon_1) + 2D\epsilon_1, c) & \textrm{ if $c$ is odd.} \\	
	\end{cases}$
	\item $(b, c) \pi_1^4 = \begin{cases}
	(b(1+4D \epsilon_1), c) & \textrm{ if $c$ is even}, \\
	(b(1-4D \epsilon_1) - 4D, c) & \textrm{ if $c$ is odd.} \\	
	\end{cases}$
	\item $(b, c) \pi_1 \pi_2 = (\ol{b}+1, -c)$. (Note that since this calculation and the definition of $\ol{b}$ is the same as in
	Theorem~\ref{t:8-2b-8}, it follows at once that $\pi_1 \pi_2$ and $\pi_1 \pi_2 \pi_3$ have order 2.)
	\item $(b, c) \pi_2^8 = (b+8, c)$. (This follows from the same calculation in Theorem~\ref{t:2b-2bp1-2b}.)
	\end{enumerate}
	Lemma~\ref{lem:perm-rep} and Corollary \ref{cor:atomic-ch-ch}
	then finish the proof.	
	\end{proof}
	
	Table~\ref{tab:atomic-4} includes information on all of the atomic chiral 4-polytopes with chiral facets and vertex-figures.

\section{Atomic chiral $4$-polytopes with regular facets and chiral vertex-figures}\label{s:ChiralRegular}

Now we switch our attention to atomic chiral $4$-polytopes with regular facets and chiral vertex-figures. The goal is to show that the vertex-figures are atomic chiral polyhedra, and then use the classifications in Section \ref{s:polyhedra} to find all atomic chiral $4$-polytopes with regular facets.

\subsection{The structure of atomic chiral $4$-polytopes with regular facets and chiral vertex-figures}

As in the previous section, we start by studying normal subgroups of the rotation group of atomic chiral $4$-polytopes, in this case with regular facets.

\begin{lemma}\label{l:MidEven}
Let $\calP$ be an atomic chiral $4$-polytope with regular facets, chiral vertex-figures and type $\{p,q,r\}$. If $\Gamma(\calP)=\langle \s_1,\s_2, \s_3 \rangle$ then
\begin{itemize}
 \item[(a)] $\langle \s_1 \rangle$ is core-free,
 \item[(b)] $q$ is even,
 \item[(c)] $p<q$.
\end{itemize}
\end{lemma}

\begin{proof}
If $\langle \s_1^k \rangle \triangleleft \Gamma(\calP)$ then by Proposition \ref{prop:vfig-quo}, $\calP/\langle \s_1^k \rangle$ is a tight polytope with vertex-figures isomorphic to those of $\calP$. The chirality of the vertex-figures of $\calP$ contradicts atomicity, proving part (a).

To prove part (b), assume to the contrary that $q$ is odd. Since $\langle \s_1 \rangle$ is core-free, the type of the facets of $\calP$ must be the dual of one of the types listed in Theorem \ref{t:TightRegularHedra}. The only possibility for $q$ being odd is if the facets of $\calP$ have type $\{2,q\}$. This contradicts Lemma \ref{l:NoChiral2}.

Part (c) follows from part (1) and Proposition \ref{prop:NonTrivialCore}.
\end{proof}

\begin{lemma}\label{l:Nos3Quo}
Let $\calP$ be an atomic chiral $4$-polytope with regular facets and chiral vertex-figures. If $\calP$ has type $\{p,q,r\}$ then the vertex-figures of $\calP$ do not cover a chiral polyhedron with type $\{q,r'\}$ for $r'<r$.
\end{lemma}

\begin{proof}
Assume to the contrary that the vertex-figures of $\calP$ cover a chiral polyhedron $\calQ$ with type $\{q,r'\}$ with $r'<r$. Then 
$\langle \s_3^{r'} \rangle \triangleleft \langle \s_2,\s_3 \rangle$, and $\Gamma(\calQ)=\langle \s_2,\s_3 \rangle / \langle \s_3^{r'}\rangle$. In particular, $\langle \s_3^{r'} \rangle$ is normalized by conjugation by $\s_2$. The dual version of Lemma \ref{l:rels1s3} implies that it is also normalized by conjugation by $\s_1$ and hence $\langle \s_3^{r'} \rangle \triangleleft \Gamma(\calP)$.

By Proposition \ref{prop:vfig-quo}, $\calP/\langle \s_3^{r'} \rangle$ is a $4$-polytope. Furthermore, its vertex-figures are isomorphic to $\calQ$, which is chiral. This contradicts the atomicity of $\calP$. 
\end{proof}

\begin{lemma}\label{l:Nos3Quo2}
Let $\calP$ be an atomic chiral $4$-polytope with type $\{p,q,r\}$, regular facets and chiral vertex-figures. Then the vertex-figures of $\calP$ do not cover a chiral polyhedron with type $\{q',r\}$ with either $q'$ an even divisor of $q$ or $q' < q/2$.
\end{lemma}

\begin{proof}
Let $\Gamma(\calP) = \langle \s_1, \s_2, \s_3 \rangle$.

Assume first that the vertex-figures of $\calP$ cover a chiral polyhedron $\calQ$ with type $\{q',r\}$ with $q'$ even. Then $\langle \sigma_2^{q'} \rangle \triangleleft \langle \s_2,\s_3 \rangle$.
By the dual version of Lemma \ref{l:sigma2IsNormal}, $\langle \s_2^2 \rangle \triangleleft \langle \s_1,\s_2 \rangle$. Since $\langle \s_2^{q'} \rangle \le \langle \s_2^2 \rangle$ and the latter is cyclic, we have that $\langle \s_2^{q'} \rangle \triangleleft \Gamma(\calP)$. Then Proposition \ref{prop:vfig-quo} shows that $\calP/\langle \s_2^{q'} \rangle$ is a $4$-polytope whose vertex-figures are isomorphic to $\calQ$, contradicting atomicity of $\calP$.

Now, if $q'$ is odd and $q' < q/2$ then $\langle \sigma_2^{2q'} \rangle$ is invariant under conjugation by all generators $\s_i$ and hence it is a proper normal subgroup of $\Gamma(\calP)$. It follows that $\calP/\langle \sigma_2^{2q'} \rangle$ is a proper quotient of $\calP$ whose vertex-figures cover $\calQ$. Proposition \ref{prop:RegularsNotCoverChirals} implies that $\calP/\langle \sigma_2^{2q'} \rangle$ is a chiral quotient of $\calP$, again contradicting atomicity of $\calP$.
\end{proof}

	We are now ready to prove the main necessary condition for a tight chiral $4$-polytope with regular facets and chiral vertex-figures to be atomic.

	\begin{theorem} \label{t:reg-ch-atomics}
	If $\calP$ is an atomic chiral $4$-polytope with regular facets and chiral vertex-figures then the vertex-figures are atomic chiral polyhedra.
	\end{theorem}

\begin{proof}
Let $\Gamma(\calP)=\langle \s_1,\s_2,\s_3 \rangle$ and assume that the facets and vertex-figures of $\calP$ are isomorphic to $\calQ_1$ and $\calQ_2$, respectively. We shall abuse notation and write $\Gamma^+(\calQ_1) = \langle \s_1,\s_2 \rangle$ and $\Gamma(\calQ_2) = \langle \s_2,\s_3 \rangle$.

Assume to the contrary that $\calQ_2$ is not atomic. Lemmas \ref{l:Nos3Quo} and \ref{l:Nos3Quo2} imply that $\calQ_2$ covers no chiral polyhedron with type $\{q,r'\}$ with $r'<r$, and that the only chiral polyhedron covered by $\calQ_2$ with type $\{q',r\}$ for $q'<q$ is such that $q'=q/2$. Furthermore, $q/2$ must be odd.

First, we show that $\calQ_2 / \langle \s_2^{q/2} \rangle$ is atomic. Note that since $\s_2^{q/2}$ has order $2$ and $\langle \s_2^{q/2} \rangle \triangleleft \Gamma(\calQ_2)$, $\s_2^{q/2}$ is central in $\Gamma(\calQ_2)$. Let $\Gamma(\calQ_2 / \langle \s_2^{q/2} \rangle) = \langle \hat{\s}_2,\hat{\s}_3 \rangle$. By Lemma \ref{l:Nos3Quo2}, $\calQ_2 / \langle \s_2^{q/2} \rangle$ does not cover a chiral polyhedron with type $\{q'',r\}$ for $q''<q/2$.
On the other hand, if $\calQ_2 / \langle \s_2^{q/2} \rangle$ covers a chiral polyhedron with type $\{q/2,r'\}$ for $r'<r$ then $r'$ must be even and $\hat{\s}_2^{-1} \hat{\s}_3^{r'} \hat{\s}_2 = \hat{\s}_3^{ar'}$ for some integer $a$. Lifting this relation to $\Gamma(\calQ_2)$ we have that $\s_2^{-1} \s_3^{r'} \s_2 = \s_2^{\epsilon q/2}\s_3^{ar'}$ for some $\epsilon \in \{0,1\}$; however, by Lemma \ref{l:Nos3Quo} $\langle \s_3^{r'} \rangle$ is not normal in $\Gamma(\calQ_2)$ and hence $\epsilon=1$. Then, conjugation by $\s_2$ interchanges the subgroups $\langle \s_3^{r'} \rangle$ and $\langle \s_2^{q/2} \s_3^{r'} \rangle$, implying that $\s_2^{-\ell} \s_3^{r'} \s_2^\ell \in \langle \s_3^{r'} \rangle$ if and only if $\ell$ is even. It follows that
$\s_3^{r'} = \s_2^{q/2} \s_3^{r'} \s_2^{q/2} \notin \langle \s_3^{r'} \rangle$, a contradiction. Therefore $\calQ_2 / \langle \s_2^{q/2} \rangle$ is atomic.

Since $\calQ_2 / \langle \s_2^{q/2} \rangle$ is atomic and has type $\{q/2,r\}$ with $q/2$ odd, we have that $q/2 = m^\beta$ and $r=2m$ for some odd prime $m$ and positive integer $\beta$. In particular $q=2m^\beta$ and, by Theorem \ref{t:TightRegularHedra}, $p$ must be the odd prime power $m^\beta$. Furthermore, by \cite[Prop. 3.2 and Thm. 3.6]{tight-chiral-polyhedra}, the atomic chiral polyhedron of type $\{m^{\beta}, 2m\}$ covers a tight regular polyhedron of type $\{m, 2m\}$, and so $\langle \s_2^m \rangle$ is normal in $\langle \s_2, \s_3 \rangle / \langle \s_2^{q/2} \rangle$ and indeed in $\langle \s_2, \s_3 \rangle$ itself. Then, since the dual version of Lemma \ref{l:sigma2IsNormal} tells us that $\langle \s_2^2 \rangle$ is normal in $\langle \s_1, \s_2 \rangle$, it follows that $\langle \s_2^{2m} \rangle$ is normal in $\G(\calP)$.

Abusing notation let $\Gamma^+(\calQ_2/\langle \s_2^{m} \rangle)=\langle \s_2',\s_3 \rangle$. Since $m$ is odd and $\calQ/\langle \s_2^{m} \rangle$ is regular, Proposition \ref{prop:normal-props} and the dual version of Lemma \ref{l:sigma2IsNormal} imply that $\s_3^2$ is central in $\Gamma^+(\calQ_2/\langle \s_2^{m} \rangle)$. If $\Gamma^+(\calP/\langle \s_2^{2m} \rangle) = \langle \s_1,\s_2'',\s_3 \rangle$ then
		\[ \s_2'' \s_3^2 (\s_2'')^{-1}
		= \s_3^2 (\s_2'')^{\varepsilon m} \]
for some $\varepsilon \in \{0,1\}$.
		Now, $(\s_2'')^{m}$ generates a normal subgroup of order 2, and is thus central. Then
		\[ \id = \s_2'' \s_3^{2m} (\s_2'')^{-1} = (\s_3^2 (\s_2'')^{\varepsilon m})^{m} = \s_3^{2m} (\s_2'')^{\varepsilon m^2}
		= (\s_2'')^{\varepsilon m^2}. \]
		Since $\s_2''$ has order $2m$ and $m^2$ is odd, it follows that $\varepsilon = 0$, so in fact, $\s_3^2$ commutes with $\s_2''$. Then, by (\ref{eq:magicalRelation}) we have that
		\[ \s_1^{-1} \s_3^2 \s_1 = ((\s_2'')^2 \s_3)^2 = \s_2'' \s_3^{-2} (\s_2'')^{-1} = \s_3^{-2}, \]
		and so conjugation by $\s_1$ inverts $\s_3^2$. Since $p=m^\beta$ is odd, this implies that $\s_3^2 = \s_3^{-2}$, and so $2m$ (the order of $\s_3$) divides $4$, which is impossible.
\end{proof}

	\begin{corollary}
	\label{cor:atomic-reg-ch}
	A tight chiral $4$-polytope $\calP$ of type $\{p, q, r\}$ with regular facets and vertex-figures is atomic if and only if
	\begin{enumerate}
	\item The vertex-figures are atomic,
	\item $q$ is even, and 
	\item $\langle \s_1 \rangle$ is core-free in $\G(\calP)$.
	\end{enumerate}
	\end{corollary}
	
	\begin{proof}
	Theorem \ref{t:reg-ch-atomics} and Lemma \ref{l:MidEven} prove that the conditions are necessary. To prove that they suffice,
	suppose that $\calP$ satisfies the three conditions, and suppose that $\calP$ properly covers a chiral $4$-polytope $\calQ$.
	Then the facets of $\calQ$ are covered by the regular facets of $\calP$, and by Proposition \ref{prop:RegularsNotCoverChirals},
	the facets of $\calQ$ are regular. Then by Proposition \ref{prop:facvfChiral}, the vertex-figures of $\calQ$ are chiral.
	These vertex-figures are covered by the vertex-figures of $\calP$, which are atomic, and so $\calQ$ has the same vertex-figures
	as $\calP$. In particular, $\calQ$ has type $\{p', q, r\}$ for some $p'$ dividing $p$. By tightness,
	$|\G(\calQ)| = p'qr$ and $|\G(\calP)| = pqr$, and since $\calQ$ is a proper quotient of $\calP$, we have $p' < p$.
	Now, the kernel of the natural epimorphism from $\G(\calP)$ to $\G(\calQ)$
	includes $\s_1^{p'}$. On the other hand, $|\langle \s_1^{p'} \rangle| = p/p'$ so that $|\G(\calP)| = |\G(\calQ)| \cdot
	|\langle \s_1^{p'} \rangle|$. It follows that $\s_1^{p'}$ generates a nontrivial normal subgroup of $\G(\calP)$,
	contradicting that $\langle \s_1 \rangle$ is core-free. So $\calP$ must be atomic.
	\end{proof}

\subsection{Classification of atomic chiral $4$-polytopes with regular facets and chiral vertex-figures}

	If $\calP$ is an atomic chiral $4$-polytope with regular facets and chiral vertex-figures, then Lemma~\ref{l:MidEven}
	implies that the facets must be the dual of one of the polyhedra in Theorem~\ref{t:TightRegularHedra}, and
	Theorem~\ref{t:reg-ch-atomics} implies that the vertex-figures must be one of the polyhedra in Theorem~\ref{t:atomic-polyhedra}
	or its dual. The dual of Lemma~\ref{l:NoChiral2} implies that the facets cannot be type $\{2, q\}$. Then, after some manipulation of the relations in
	\cite[Section 4]{tight3} we have the following lemma:
	
	\begin{lemma} \label{l:tight-reg-facets}
	The facets of an atomic chiral $4$-polytopes with regular facets must be one of the following:
	\begin{enumerate}
	\item Type $\{m, 2m\}$ for an odd prime $m$, with rotation group $[m, 2m]^+ / (\s_2^2 \s_1 = \s_1 \s_2^2)$.
	\item Type $\{4, 8\}$, with rotation group $[4, 8]^+ / (\s_2^2 \s_1 = \s_1 \s_2^2)$.
	\item Type $\{4, 2^{\beta}\}$ for some $\beta \geq 5$, with rotation group $[4, 2^{\beta}]^+ / (\s_2^{-1} \s_1 = \s_1^{-1} \s_2^{1+2^{\beta-1}})$.
	\item Type $\{2^{\beta-1}, 2^{\beta}\}$ for some $\beta \geq 5$, with rotation group $[2^{\beta-1}, 2^{\beta}]^+ / (\s_2 \s_1^{-1} = \s_1 \s_2^{3-\epsilon 2^{\beta-1}})$, with $\epsilon \in \{0, 1\}$.
	\end{enumerate}
	\end{lemma}

	\begin{table}[htbp]
	\centering
	\begin{tabular}{|l|l|l|l|} \hline
	Group name & Parent group & Extra relations & Notes \\ \hline
	$\Lambda_1$ & $[m, 2m, m^{\alpha}]^+$, & $	\s_2^2 \s_1 = \s_1 \s_2^2$ & $m$ odd prime, $\alpha \geq 2$, \\
	& & $\s_3 \s_2^2 = \s_2^2 \s_3^{1+km^{\alpha-1}}$ & $1 \leq k \leq m-1$ \\ \hline

	$\Lambda_2$ & $[4, 8, 2^{\beta}]^+$ & $\s_2^2 \s_1 = \s_1 \s_2^2$ & $\beta \geq 5, \epsilon = \pm 1$ \\
	& & $\s_3 \s_2^2 = \s_2^2 \s_3^{1+\epsilon 2^{\beta-2}}$ & \\ \hline

	$\Lambda_3$ & $[4, 2^{\beta-1}, 2^{\beta}]^+$ & $\s_2^{-1} \s_1 = \s_1^{-1} \s_2^{1+2^{\beta-2}}$ & $\beta \geq 5, \epsilon = \pm 1$ \\
	& & $\s_3^{-1} \s_2 = \s_2^{-1+2^{\beta-2}} \s_3^{-3+ \epsilon 2^{\beta-2}}$ & \\
	& & $\s_3 \s_2^{-1} = \s_2^{1+2^{\beta-2}} \s_3^{3+ \epsilon 2^{\beta-2}}$ & \\ \hline

	$\Lambda_4$ & $[4, 2^{\beta}, 8]^+$ & $\s_2^{-1} \s_1 = \s_1^{-1} \s_2^{1+2^{\beta-1}}$ & $\beta \geq 5, \epsilon = \pm 1$ \\
	& & $\s_2 \s_3^2 = \s_3^2 \s_2^{1+\epsilon 2^{\beta-2}}$ & \\ \hline

	$\Lambda_5$ & $ [4, 2^{\beta}, 2^{\beta-1}]^+$ & $\s_2^{-1} \s_1 = \s_1^{-1} \s_2^{1+2^{\beta-1}}$ & $\beta \geq 5, \epsilon = \pm 1$ \\
	& & $\s_3^{-1} \s_2 = \s_2^{3+\epsilon 2^{\beta-2}} \s_3^{1-2^{\beta-2}}$ & \\
	& & $\s_3 \s_2^{-1} = \s_2^{-3+\epsilon 2^{\beta-2}} \s_3^{-1+2^{\beta-2}}$ &  \\ \hline

	$\Lambda_6$ & $ [2^{\beta-1}, 2^{\beta}, 8]^+$ & $\s_2 \s_1^{-1} = \s_1 \s_2^{3-\epsilon_1 2^{\beta-1}}$ & $\beta \geq 5$, \\
	& & $\s_2 \s_3^2 = \s_3^2 \s_2^{1+\epsilon_2 2^{\beta-2}}$ & $\epsilon_1 \in \{0,1\}, \epsilon_2 = \pm 1$ \\ \hline

	$\Lambda_7$ & $[2^{\beta-1}, 2^{\beta}, 2^{\beta-1}]^+$ & $\s_2 \s_1^{-1} = \s_1 \s_2^{3-\epsilon_1 2^{\beta-1}}$ & $\beta \geq 5$, \\
	& & $\s_3^{-1} \s_2 = \s_2^{3+\epsilon_2 2^{\beta-2}} \s_3^{1-2^{\beta-2}}$ & $\epsilon_1 \in \{0,1\}, \epsilon_2 = \pm 1$ \\
	& & $\s_3 \s_2^{-1} = \s_2^{-3+\epsilon_2 2^{\beta-2}} \s_3^{-1+2^{\beta-2}}$ & \\ \hline

	$\Lambda_8$ & $[2^{\beta-2}, 2^{\beta-1}, 2^{\beta}]^+$ & $\s_2 \s_1^{-1} = \s_1 \s_2^{3-\epsilon_1 2^{\beta-2}}$ & $\beta \geq 5$, \\
	& & $\s_3^{-1} \s_2 = \s_2^{-1+2^{\beta-2}} \s_3^{-3+ \epsilon_2 2^{\beta-2}}$ & $\epsilon_1 \in \{0,1\}, \epsilon_2 = \pm 1$ \\
	& & $\s_3 \s_2^{-1} = \s_2^{1+2^{\beta-2}} \s_3^{3+ \epsilon_2 2^{\beta-2}}$ & \\ \hline
	\end{tabular}
	\caption{The possible groups of atomic chiral $4$-polytopes with regular facets and chiral vertex-figures}
	\label{tab:atomic-rc}
	\end{table}

	Now there are eight possibilities for the automorphism group of an atomic chiral $4$-polytope with regular facets and chiral
	vertex-figures; see Table~\ref{tab:atomic-rc}.
	
	We will show that the first three groups do correspond to atomic chiral $4$-polytopes, whereas the remaining groups
	do not.
	
	\begin{theorem} \label{t:m-2m-ma}
	The group $\Lambda_1$ is the automorphism group of an atomic chiral polytope of type $\{m, 2m, m^{\alpha}\}$,
	for each $k$ satisfying $1 \leq k \leq m-1$.
	\end{theorem}
	
	\begin{proof}
	Let $D = k m^{\alpha-1}$. 
	Then we define permutations
	of $\Z_{2m} \times \Z_{m^{\alpha}}$ as follows:
	\begin{align*}
	(b, c) \pi_1 &= \begin{cases}
		(b + 2c, c + \frac{c(c-1)}{2} D) & \text{ if $b$ is even}, \\
		(b + 2c - 2, c + \frac{c(c-1)}{2} D) & \text{ if $b$ is odd}, \\
		\end{cases} \\
	(b, c) \pi_2 &= (b + 1 - 2c, -c + \frac{c(c-1)}{2} D), \\
	(b, c) \pi_3 &= (b, c+1).
	\end{align*}
	
	Here are a few intermediate calculations.
	\begin{enumerate}
	\item For all $n$, $(b,c) \pi_1^n = \begin{cases}
		(b + 2nc, c + n\frac{c(c-1)}{2} D) & \text{ if $b$ is even}, \\
		(b + 2nc - 2n, c + n\frac{c(c-1)}{2} D) & \text{ if $b$ is odd}. \\
	\end{cases}$
	\item $(b,c) \pi_2^2 = (b+2, c(1+D))$
	\item $(b,c) \pi_1 \pi_2 = \begin{cases}
	(b+1, -c) & \text{ if $b$ is even }, \\
	(b-1, -c) & \text{ if $b$ is odd. }
	\end{cases}$.
	\end{enumerate}
	From these, it is routine to show that there is a well-defined epimorphism from $\Lambda_1$ to $\langle \pi_1, \pi_2, \pi_3 \rangle$ sending each $\s_i$ to $\pi_i$. Lemma~\ref{lem:perm-rep} and Corollary \ref{cor:atomic-reg-ch}
	then finish the proof.	

	\end{proof}
	
	\begin{theorem} \label{t:4-8-2b}
	The group $\Lambda_2$ is the automorphism group of an atomic chiral polytope of type $\{4, 8, 2^{\beta}\}$.
	\end{theorem}
	
	\begin{proof}
	Let $D = 2^{\beta-3}$. 
	For $b \in \Z_{8}$, we define 
	\[ \overline{b} = \begin{cases}
	b & \textrm{ if $b$ is even}, \\
	b-2 & \textrm{ if $b$ is odd.}
	\end{cases} \]
	Then we define permutations
	of $\Z_{8} \times \Z_{2^{\beta}}$ as follows:
	\begin{align*}
	(b, c) \pi_1 &= \begin{cases}
	(\overline{b} - 2c, c(1+D\epsilon)) & \textrm{ if $c$ is even,} \\
	(\overline{b} - 2c + 4, c(1 - D \epsilon) + D\epsilon) & \textrm{ if $c$ is odd,} \\
	\end{cases} \\
	(b, c) \pi_2 &= \begin{cases}
	(b + 1 - 2c, c(-1 + D \epsilon)) & \textrm{ if $c$ is even,} \\
	(b + 1 - 2c, c(-1 - D \epsilon) + D\epsilon) & \textrm{ if $c$ is odd,} \\
	\end{cases} \\
	(b, c) \pi_3 &= (b, c+1).
	\end{align*}
	We note that
	\[ (b,c) \pi_2^2 = \begin{cases}
	(b+2, c(1-2D\epsilon) & \text{ if $c$ is even}, \\
	(b+2, c(1+2D\epsilon) & \text{ if $c$ is odd}.
	\end{cases} \]
	Then it is routine to show that there is a well-defined epimorphism from $\Lambda_2$ to $\langle \pi_1, \pi_2, \pi_3 \rangle$ sending each $\s_i$ to $\pi_i$. Lemma~\ref{lem:perm-rep} and Corollary \ref{cor:atomic-reg-ch}
	then finish the proof.	

	\end{proof}

	\begin{theorem} \label{t:4-2b-2b}
	The group $\Lambda_3$ is the automorphism group of an atomic chiral polytope of type $\{4, 2^{\beta-1}, 2^{\beta}\}$.
	\end{theorem}

	\begin{proof}
	Let $D = 2^{\beta-3}$. 
	For $b \in \Z_{2^{\beta-1}}$, we define 
	\[ \overline{b} = \begin{cases}
	b(-1-D) & \textrm{ if $b$ is even}, \\
	b(-1-D)+D & \textrm{ if $b$ is odd.}
	\end{cases} \]
	Then we define permutations
	of $\Z_{2^{\beta-1}} \times \Z_{2^{\beta}}$ as follows:
	\begin{align*}
	(b, c) \pi_1 &= \begin{cases}
	(\overline{b} - cD, c(-1+D\epsilon)) & \textrm{ if $c$ is even,} \\
	(\overline{b} - (c-1)D + 2, (1-c)(1+D\epsilon) + 1) & \textrm{ if $c$ is odd,} \\
	\end{cases} \\
	(b, c) \pi_2 &= \begin{cases}
	(b + 1 + cD, c(1+D\epsilon)) & \textrm{ if $c$ is even,} \\
	(b - 1 + (c-1)D, (c-1)(1-D\epsilon)-1 & \textrm{ if $c$ is odd,} \\
	\end{cases} \\
	(b, c) \pi_3 &= (b, c+1).
	\end{align*}
	Here are some intermediate calculations:	
	\begin{enumerate}
	\item If $a$ is even, then $\ol{a+b} = \ol{a} + \ol{b}$.
	\item $\overline{\overline{b}} \equiv b(1+2D)$ (mod $2^{\beta-1}$).
	\item $(b, c) \pi_1^2 = \begin{cases}
	(b(1+2D), c(1-2D\epsilon)) & \textrm{ if $c$ is even,} \\
	(b(1+2D)-2D, (c-1)(1+2D\epsilon)+1) & \textrm{ if $c$ is odd,} \\
	\end{cases}$
	\item $(b, c) \pi_2^8 = \begin{cases}
	(b+8, c), & \textrm{ if $c$ is even}, \\
	(b-8, c-16), & \textrm{ if $c$ is odd.}
	\end{cases}$
	\item $(b,c) \pi_1 \pi_2 = (\ol{b}+1, -c)$
	\end{enumerate}
	Let us rewrite the first extra relation of $\Lambda_3$ as $\s_2 \s_1^2 = \s_1^2 \s_2^{1+2D}$ (see the proof of
	Theorem \ref{t:atomic-polyhedra}, noting that $\s_1^{-1} = \s_1^3$). Similarly, we rewrite the second extra
	relation by multiplying both sides on the left by $\s_2$, and the third relation by multiplying both sides on
	the right by $\s_2$. 
	Then one can check that there is a well-defined epimorphism from $\Lambda_3$ to $\langle \pi_1, \pi_2, \pi_3 \rangle$ 
	sending each $\s_i$ to $\pi_i$. Lemma~\ref{lem:perm-rep} and Corollary \ref{cor:atomic-reg-ch}
	then finish the proof.	
	\end{proof}
	
	In order to rule out the remaining cases, we will use the following lemma.
	
	\begin{lemma} \label{l:forced-rel}
	Suppose that $\Lambda$ is a quotient of $[p, q, r]^+$ satisfying
	\begin{align*}
	\s_2^2 \s_1 &= \s_1 \s_2^{2t}, \\
	\s_3^{-1} \s_2 &= \s_2^a \s_3^c, \\
	\s_3 \s_2^{-1} &= \s_2^b \s_3^{-c},
	\end{align*}
	and suppose that $b$ is odd. Then
	\[ \s_2^{2t+2} \s_3 = \s_3 \s_2^{-t(b+1)-1+a}. \]
	\end{lemma}
	
	\begin{proof}
	First, we note that
	\[ \s_1^{-1} \s_2^2 \underline{\s_3 \s_1} = \s_1^{-1} \underline{\s_2^2 \s_1} \s_2^2 \s_3 = \s_2^{2t+2} \s_3. \]
	On the other hand,
	\begin{align*}
	\s_1^{-1} \underline{\s_2^2 \s_3} \s_1 &= \s_1^{-1} \underline{\s_2 \s_3^{-1}} \s_2^{-1} \s_1 \\
	&= \s_1^{-1} \s_3^c \underline{\s_2^{-b-1} \s_1} \\
	&= \underline{\s_1^{-1} \s_3^c \s_1} \s_2^{-t(b+1)} \\
	&= \underline{\s_2 \s_3^{-c}} \s_2^{-t(b+1)-1} \\
	&= \s_3 \s_2^{-t(b+1)-1+a}.
	\end{align*}
	\end{proof}
	
	\begin{theorem} \label{t:no-L4-L5}
	The groups $\Lambda_4$ and $\Lambda_5$ are not the automorphism groups of tight chiral $4$-polytopes.
	\end{theorem}
	
	\begin{proof}
	In $\Lambda_4$ and $\Lambda_5$, the relation $\s_2^{-1} \s_1 = \s_1^{-1} \s_2^{1+2^{\beta-1}}$ is equivalent to $\s_2 \s_1^{-1} = \s_1 \s_2^{-1+2^{\beta-1}}$ (see \cite[Proposition 3.1]{tight3}), and this implies that
	$\s_2^2 \s_1 = \s_2 \s_1^{-1} \s_2^{-1} = \s_1 \s_2^{2(-1+2^{\beta-2})}$. Then Lemma~\ref{l:forced-rel} proves that both groups
	satisfy $\s_2^{2^{\beta-1}} \s_3 = \s_3$, and so $\s_2$ does not have order $2^{\beta}$ as required.
	\end{proof}
	
	\begin{theorem} \label{t:no-L6-L7}
	The groups $\Lambda_6$ and $\Lambda_7$ are not the automorphism groups of tight chiral $4$-polytopes.
	\end{theorem}
	
	\begin{proof}
	If either group is the automorphism group of a tight chiral $4$-polytope, then Proposition~\ref{prop:NonTrivialCore} implies that 
	$\langle \s_1, \s_2 \rangle$ and $\langle \s_2, \s_3 \rangle$ both have a normal subgroup of the form $\langle \s_2^k \rangle$. 
	It follows that $\langle \s_2^{2^{\beta-1}} \rangle$ is normal in $\langle \s_1, \s_2, \s_3 \rangle$, which means
	that $\s_2^{2^{\beta-1}}$ is central.

	Now, the relation $\s_2 \s_1^{-1} = \s_1 \s_2^{3-\epsilon_1 2^{\beta-1}}$ implies that $\s_2^2 \s_1 = 
	\s_1 \s_2^{2(1+\epsilon_1 2^{\beta-2})}$. Then Lemma~\ref{l:forced-rel} implies that in $\Lambda_6$, 
	$\s_2^{4+\epsilon_1 2^{\beta-1}} \s_3 = \s_3 \s_2^{-4-\epsilon_1 2^{\beta-1}}$.
	So conjugation by $\s_3$ inverts $\s_2^{4+\epsilon_1 2^{\beta-1}}$, and since $\s_2^{2^{\beta-1}}$ is central, this implies 
	that conjugation by $\s_3$ inverts $\s_2^4$. 
	Now, $\s_3^2 \s_2 = \s_3 \s_2^{-1} \s_3^{-1} = \s_2^{1+\epsilon_2 2^{\beta-2}} \s_3^2$, and it follows that
	$\s_3^4 \s_2 = \s_2^{1+\epsilon_2 2^{\beta-1}} \s_3^4$. Then $\s_3^4 \s_2^2 = \s_2^{2(1+\epsilon_2 2^{\beta-1})} \s_3^4 = \s_2^2 \s_3^4$,
	and since $\s_3$ has order 8, this implies that $\s_2^2 = \s_3^4 \s_2^2 \s_3^4$. So:
	\begin{align*}
	\s_3 \s_2^4 &= \s_2^{-1} \underline{\s_3^{-1} \s_2} \s_2^2 \\
	&= \s_2^{-2+\epsilon_2 2^{\beta-2}} \underline{\s_3^{-3} \s_2^2} \\
	&= \s_2^{-2+\epsilon_2 2^{\beta-2}} \underline{\s_3 \s_2^2} \s_3^4 \\
	&= \s_2^{-2+\epsilon_2 2^{\beta-2}} \s_2^{-2+\epsilon_2 2^{\beta-2}} \s_3^{-3} \s_3^4 \\
	&= \s_2^{-4+\epsilon_2 2^{\beta-1}} \s_3.
	\end{align*}
	Since we also have that conjugation by $\s_3$ inverts $\s_2^4$, this implies that $\s_2^{2^{\beta-1}} = \id$, and so $\s_2$ does not
	have the desired order.

	In $\Lambda_7$, Lemma~\ref{l:forced-rel} implies that $\s_2^{4+\epsilon_1 2^{\beta-1}} \s_3 = \s_3 \s_2^{4+\epsilon_1 2^{\beta-1}}$.
	Since $\s_2^{2^{\beta-1}}$ is central, this implies that $\s_3$ commutes with $\s_2^4$. However,
	\begin{align*}
	\s_3 \s_2^{-4-\epsilon_2 2^{\beta-2}} &= \underline{\s_3 \s_2^{-1}} \s_2^{-3-\epsilon_2 2^{\beta-2}} \\
	&= \s_2^{-3+\epsilon_2 2^{\beta-2}} \underline{\s_3^{-1+2^{\beta-2}} \s_2^{-3-\epsilon_2 2^{\beta-2}}} \\
	&= \s_2^{-4+\epsilon_2 2^{\beta-2}} \s_3.
	\end{align*}
	Then $\s_2^{-2^{\beta-2}} = \s_2^{2^{\beta-2}}$, which implies that $\s_2^{2^{\beta-1}} = \id$, and again $\s_2$ does
	not have the desired order.
	\end{proof}

	\begin{theorem} \label{t:no-L8}
	The group $\Lambda_8$ is not the automorphism group of tight chiral $4$-polytope.
	\end{theorem}
	
	\begin{proof}
	If $\Lambda_8$ were the automorphism group of a tight chiral $4$-polytope, then $\langle \s_2 \rangle$ would be core-free in 
	$\langle \s_2, \s_3 \rangle$ (see Proposition \ref{prop:NonTrivialCore} and \cite[Proposition 4.5]{tight-chiral-polyhedra}). We will show that in fact, $\s_3$ normalizes a nontrivial subgroup
	of $\langle \s_2 \rangle$. 

	From the relation $\s_2 \s_1^{-1} = \s_1 \s_2^{3-\epsilon_1 2^{\beta-2}}$, it follows that
	$\s_2^2 \s_1 = \s_1 \s_2^{2-\epsilon_1 2^{\beta-2}}$, and thus for each $k$, $\s_2^{2k} \s_1 =
	\s_1 \s_2^{(1-\epsilon_1 2^{\beta-3})2k}$. 
	Then
	\begin{align*}
	\s_1^{-1} \s_3 \underline{\s_2^2 \s_1} &= \underline{\s_1^{-1} \s_3 \s_1} \s_2^{2 - \epsilon_1 2^{\beta-2}} \\
	&= \s_2^2 \s_3 \s_2^{2 - \epsilon_1 2^{\beta-2}}.
	\end{align*}
	On the other hand,
	\begin{align*}
	\s_1^{-1} \underline{\s_3 \s_2^2} \s_1 &= \s_1^{-1} \s_2^{-2+2^{\beta-2}} \underline{\s_3^{-3+\epsilon_2 2^{\beta-2}} \s_1} \\
	&= \underline{\s_1^{-1} \s_2^{-2+2^{\beta-2}} \s_1} \s_2 \s_3^{3-\epsilon_2 2^{\beta-2}} \s_2^{-1} \\
	&= \s_2^{-1+2^{\beta-2}+\epsilon_1 2^{\beta-2}} \s_3^{3-\epsilon_2 2^{\beta-2}} \s_2^{-1} \\
	&= \s_2^{-1+2^{\beta-2}+\epsilon_1 2^{\beta-2}} \underline{\s_3^{3- \epsilon_2 2^{\beta-2}} \s_2^{1-2^{\beta-2}}} \s_2^{-2+2^{\beta-2}} \\
	&= \s_2^{-1+2^{\beta-2}+\epsilon_1 2^{\beta-2}} \s_2^{-1} \s_3 \s_2^{-2+2^{\beta-2}}.
	\end{align*}
	Putting these together, we find that $\s_2^{-4+2^{\beta-2} + \epsilon_1 2^{\beta-2}} \s_3 = \s_3 \s_2^{4 - 2^{\beta-2} - 
	\epsilon_1 2^{\beta-2}}$, and so $\s_3$ normalizes a nontrivial subgroup of $\langle \s_2 \rangle$.
	\end{proof}
	
	Table~\ref{tab:atomic-4} summarizes the atomic chiral $4$-polytopes with chiral vertex figures. The duals of the first three rows yield atomic chiral $4$-polytopes with regular vertex-figures, and the last two rows correspond to a dual pair of chiral $4$-polytopes. In total, there are 11 families of atomic chiral $4$-polytopes. Thus we have shown:
	
	\begin{theorem}
	Every tight chiral $4$-polytope covers one of the polytopes in Table~\ref{tab:atomic-4} or its dual.
	\end{theorem}
	
	\begin{proposition} \label{p:xp-rc}
	If $\calP$ is an atomic chiral $4$-polytope with regular facets, then $X(\calP)$ is contained in $\langle \s_3 \rangle$.
	\end{proposition}
	
	\begin{proof}
	An atomic chiral $4$-polytope $\calP$ with regular facets has automorphism group $\Lambda_1$, $\Lambda_2$, or $\Lambda_3$. In each case,
	the chirality group of the vertex-figures is a cyclic group of prime order of the form $\langle \s_3^c \rangle$ that is normal in 
	$\langle \s_2, \s_3 \rangle$ (see Table \ref{tab:atomic-polyhedra}), and thus in $\G(\calP)$ (by the dual of Lemma 
	\ref{l:rels1s3}). The quotient of $\G(\calP)$ by this normal subgroup is a polytope, by Proposition \ref{prop:vfig-quo},
	and thus it is regular (by atomicity). Therefore, the chirality group of the vertex-figures contains the chirality group
	of $\calP$, and since the former has prime order and the latter is nontrivial, it follows that the two coincide,
	proving the claim.
	\end{proof}

	\begin{landscape}
	\begin{table}[htbp]
	\centering
	\begin{tabular}{l l l l l l l} \hline
	$\{p, q, r\}$ & Facets & $\s_2^{-1} \s_1$ & $\s_2 \s_1^{-1}$ & $\s_3^{-1} \s_2$ & $\s_3 \s_2^{-1}$ & Notes \\ \hline \hline
	$\{m, 2m, m^{\alpha}\}$ & Regular & $\s_1^{-1} \s_2^{-3}$ & $\s_1 \s_2^3$ & $\s_2^3 \s_3^{1 + km^{\alpha-1}}$
	& $\s_2^{-3} \s_3^{-1 + km^{\alpha-1}}$ & $m$ odd prime, $\alpha \geq 2$, \\ &&&&&&$1 \leq k \leq m-1$ \\ \hline

	$\{4, 8, 2^{\beta}\}$ & Regular & $\s_1^{-1} \s_2^{-3}$ & $\s_1 \s_2^3$ & $\s_2^3 \s_3^{1+ \epsilon 2^{\beta-2}}$
	& $\s_2^{-3} \s_3^{-1+ \epsilon 2^{\beta-2}}$ & $\beta \geq 5$, $\epsilon = \pm 1$ \\ \hline
	
	$\{4, 2^{\beta-1}, 2^{\beta}\}$ & Regular & $\s_1^{-1} \s_2^{1+2^{\beta-2}}$ & $\s_1 \s_2^{-1+2^{\beta-2}}$ & 
	$\s_2^{-1+2^{\beta-2}} \s_3^{-3+ \epsilon 2^{\beta-2}}$ & $\s_2^{1+2^{\beta-2}} \s_3^{3+ \epsilon 2^{\beta-2}}$ &
	$\beta \geq 5$, $\epsilon = \pm 1$ \\ \hline

	$\{2m, m^{\alpha}, 2m\}$ & Chiral & $\s_1^3 \s_2^{1+km^{\alpha-1}}$ & $\s_1^{-3} \s_2^{-1+km^{\alpha-1}}$ &
	$\s_2^{-1 + km^{\alpha-1}} \s_3^{-3}$ & $\s_2^{1 + km^{\alpha-1}} \s_3^3$ &
	$m$ odd prime, $\alpha \geq 2$, \\ &&&&&&$1 \leq k \leq m-1$ \\ \hline
	
	$\{8, 2^{\beta}, 8\}$ & Chiral & $\s_1^3 \s_2^{1+ \epsilon_1 2^{\beta-2}}$ & $\s_1^{-3} \s_2^{-1 + \epsilon_1 2^{\beta-2}}$ &
	$\s_2^{-1 + \epsilon_2 2^{\beta-2}} \s_3^{-3}$ & $\s_2^{1 + \epsilon_2 2^{\beta-2}} \s_3^3$ & $\beta \geq 5$, $\epsilon_1,
	\epsilon_2 = \pm 1$ \\ \hline
	
	$\{2^{\beta-1}, 2^{\beta}, 2^{\beta-1}\}$ & Chiral & $\s_1^{-1+2^{\beta-2}} \s_2^{-3+\epsilon_1 2^{\beta-2}}$
	& $\s_1^{1+2^{\beta-2}} \s_2^{3+\epsilon_1 2^{\beta-2}}$ & $\s_2^{3 + \epsilon_2 2^{\beta-2}} \s_3^{1+2^{\beta-2}}$
	& $\s_2^{-3 + \epsilon_2 2^{\beta-2}} \s_3^{-1+2^{\beta-2}}$ & $\beta \geq 5$, $\epsilon_1,
	\epsilon_2 = \pm 1$ \\ \hline
	
	$\{8, 2^{\beta}, 2^{\beta-1}\}$ & Chiral & $\s_1^3 \s_2^{1+ \epsilon_1 2^{\beta-2}}$ & $\s_1^{-3} \s_2^{-1 + \epsilon_1 2^{\beta-2}}$
	& $\s_2^{3 + \epsilon_2 2^{\beta-2}} \s_3^{1+2^{\beta-2}}$ & $\s_2^{-3 + \epsilon_2 2^{\beta-2}} \s_3^{-1+2^{\beta-2}}$
	& $\beta \geq 5$, $\epsilon_1, \epsilon_2 = \pm 1$ \\ \hline
	
	$\{2^{\beta-1}, 2^{\beta}, 8\}$ & Chiral & $\s_1^{-1+2^{\beta-2}} \s_2^{-3+\epsilon_1 2^{\beta-2}}$
	& $\s_1^{1+2^{\beta-2}} \s_2^{3+\epsilon_1 2^{\beta-2}}$ & $\s_2^{-1 + \epsilon_2 2^{\beta-2}} \s_3^{-3}$ & 
	$\s_2^{1 + \epsilon_2 2^{\beta-2}} \s_3^3$ & $\beta \geq 5$, $\epsilon_1, \epsilon_2 = \pm 1$ \\ \hline
	
	\end{tabular}
	\caption{The atomic chiral $4$-polytopes with chiral vertex-figures}
	\label{tab:atomic-4}
	\end{table}
	\end{landscape}
	
\section{Tight chiral $5$-polytopes} 
\label{s:5polytopes}

Recall that a tight chiral $5$-polytope must have chiral facets and chiral vertex-figures (see Proposition \ref{prop:facvfChiral} (c)). In this section we prove Theorem \ref{t:NoTight5Poly}, that is, that no such polytope exists. 

We say that a chiral $5$-polytope $\calP$ with $\Gamma(\calP)= \langle \s_1,\s_2,\s_3,\s_4 \rangle$ is {\em atomic} if it does not properly 
cover any tight chiral polytope. Clearly, every tight chiral $5$-polytope covers an atomic chiral $5$-polytope.

We start by giving properties that atomic chiral $5$-polytopes must satisfy, should they exist.

\begin{lemma}\label{l:2faces}
Let $\calP$ be a tight chiral $5$-polytope with type $\{p,q,r,s\}$ where $q \ge r$. Then the kernel of the action of $\Gamma(\calP)$ on the chains containing a $3$-face and a facet is non-trivial.
\end{lemma}

\begin{proof}
The stabilizer of the chain containing the base $3$-face and the base facet is $\Delta = \langle \sigma_1, \sigma_2 \rangle$. The remaining chains can be associated to right cosets of $\Delta$. Proposition~\ref{prop:NonTrivialCore} implies that there is a nontrivial subgroup $\langle \s_2^k \rangle$ that is normal in $\langle \s_2, \s_3, \s_4 \rangle$. Then it follows that for all $a$ and $b$ we have $(\langle \s_1, \s_2 \rangle \s_3^a \s_4^b) \s_2^k = \langle \s_1, \s_2 \rangle \s_3^a \s_4^b$, and so $\s_2^k$ fixes all chains containing a $3$-face and a facet.
\end{proof}

\begin{lemma}\label{l:5kernels}
Let $\calP$ be an atomic chiral $5$-polytope with $\Gamma(\calP) = \langle \sigma_1, \s_2, \s_3, \sigma_4\rangle$ and type $\{p, q,r, s\}$. If $q \ge r$ then
\begin{enumerate}
 \item $X(\calP)$ is $\langle \sigma_2^{q'} \rangle$ for some $q'$ satisfying that $q/q'$ is prime,
 \item The chirality groups of the base facet and the base vertex-figure are also $\langle \sigma_2^{q'} \rangle$.
\end{enumerate} 
\end{lemma}

\begin{proof}
Let $H$ and $K$ be the kernels of the actions of $\Gamma(\calP)$ on the vertices and on the chains consisting of a $3$-face and a $4$-face, respectively. By Corollary \ref{cor:NTkernel} and Lemma \ref{l:2faces}, $H$ and $K$ are non-trivial. Therefore $H \cap K$ is a normal subgroup of $\Gamma(\calP)$ that by the intersection condition is contained in $\langle \sigma_2 \rangle$. The rest of the proof is as in Proposition \ref{prop:chirGroups}.
\end{proof}

Now we can prove Theorem \ref{t:NoTight5Poly}.

\begin{proof}[Proof of Theorem \ref{t:NoTight5Poly}]
It suffices to show that there are no atomic chiral $5$ polytopes. Suppose to the contrary that $\calP$ is an atomic chiral
$5$-polytope. Up to duality, we may assume that $q \geq r$. Let $\calK$ be the base facet. It must be a tight chiral $4$-polytope,
by Proposition \ref{prop:facvfChiral}(b), and since the facets of the facets of a chiral polytope are always regular, $\calK$ has regular
facets. Now, Lemma~\ref{l:5kernels} says that $\calK$ has chirality group contained in $\langle \s_2 \rangle$. Let $\calK'$
be an atomic chiral $4$-polytope that is covered by $\calK$, with $\G(\calK') = \langle \s_1', \s_2', \s_3' \rangle$. 
Then Lemma \ref{l:ChirGroupS2} says that $X(\calK')$ is contained in $\langle \s_2' \rangle$, which contradicts
Proposition~\ref{p:xp-rc}.
\end{proof}
	
\section{Concluding remarks}

The study of tight chiral polytopes was originated in the search for chiral polytopes with a small number of flags. In ranks $3$ and $4$ the atomic chiral polytopes are now classified; this constitutes the first step for a full classification of tight chiral $3$- and $4$-polytopes. However, the techniques used to classify tight regular polyhedra fail in the chiral setting, and the full classification seems to require several more steps.

The non-existence of tight chiral $n$-polytopes for $n \ge 5$ strengthens the general belief that for each $n \ge 5$ the chiral $n$-polytopes with the fewest flags have considerably more flags than the regular $n$-polytopes with the fewest flags. See also \cite[Theorem 5.5]{non-flat}.

\bibliographystyle{amsplain}
\bibliography{gabe}

\end{document}